\documentclass[a4paper]{amsart}
\usepackage{common2}
\usepackage[latin1]{inputenc}
\usepackage{ae,aecompl,amsbsy,amssymb,amsmath,amsthm,
eurosym,amsfonts,epsfig,graphicx,graphics,verbatim,esint,color,MnSymbol}
\usepackage{enumitem}

\allowdisplaybreaks

\usepackage[backend=biber, style=alphabetic, giveninits=true, sorting=ynt]{biblatex}
\usepackage{hyperref}
\hypersetup{colorlinks=true,linkcolor=blue}

\numberwithin{equation}{section}

\title{Testing compactness of linear operators}

\author{Timo S. H\"anninen}
\author{Tuomas Oikari}

\address{Timo S. H\"anninen, Mathematics Research Centre, Tampere University, Finland}
\email{timo.s.hanninen@tuni.fi}

\address{Tuomas Oikari, Department of Mathematics and Systems Analysis, Aalto University, Finland}
\email{tuomas.oikari@aalto.fi}

\begin{document}

    \begin{abstract}
 Let $(F_i)$ be a sequence of sets in a Banach space $X$. For what sequences does the condition
 $$
\lim_{i\to \infty} \sup_{f_i\in F_i} \|Tf_i\|_Y=0
$$
hold for every Banach space $Y$ and every compact operator $T:X\to Y$? We answer this question by giving sufficient (and necessary) criteria for such sequences. We illustrate the applicability of the criteria by examples from the literature and by characterizing the  $L^p\to L^p$ compactness of dyadic paraproducts on general measure spaces through introducing a space of vanishing mean oscillations with respect to mass (as opposed to length) of cubes. 
\end{abstract}

\maketitle
\tableofcontents

\section{Introduction}
The idea of testing conditions for operators, also known as $T(1)$ theorems, is to deduce the boundedness of an operator by testing its boundedness against a class of test functions, or to deduce the compactness of 
an operator
by testing its compactness against a class of test sequences. 
This vast active area of harmonic analysis dates back to David and Journ\'e's $T(1)$ theorem for boundedness of Calder\'on--Zygmund singular integrals 
\cite{david1984}
and to Uchiyama's $T(1)$ theorem for compactness of commutators of Calder\'on--Zygmund operators and pointwise multiplication 
\cite{Uch1978}.

Let $T:X\to Y$ be a compact linear operator from a Banach space $X$ into another Banach space $Y$. The necessary and sufficient conditions for compactness of operators in harmonic analysis, such as commutators (e.g. Uchiyama \cite{Uch1978}) and Calder\'on--Zygmund operators (e.g. Villarroya \cite{Vil2015}), have the form (or are derived from conditions of the form)
$$
\limsup_{i\to \infty} \sup_{f_i\in F_i} \norm{Tf_i}_Y=0\quad\text{for a certain sequence $(F_i)$ of sets in $X$.}
$$
Which classes of test sequences are admissible for testing compactness in the sense that such a condition is necessary for the compactness of any linear operator?  In other words, for what sequences such a condition is necessary for compactness universally, independent of the precise nature of the compact operator $T:X\to Y$?

\begin{definition}[Admissibility]A sequence $(F_i)$ of sets in a Banach space $X$ is called {\it admissible for testing compactness} (shortened to {\it admissible}) if 
$$
\limsup_{i\to \infty} \sup_{f_i\in F_i} \norm{Tf_i}_Y=0
$$
for every Banach space $Y$ and every compact linear operator $T:X\to Y$.
\end{definition}
This paper has the following objectives: 
\begin{itemize}
\item Characterize admissible sequences.
\item Connect admissibility with the geometry of Banach spaces.
\item Give easily verifiable admissibility criteria.
\item Illustrate the applicability of the admissibility criteria for $T(1)$ theorems for compactness of operators. 
\end{itemize}

First, we connect admissibility of sequences to weakly nullness (Section \ref{section:weaklynull}) and Ces\`aro summability of sequences (Section \ref{section:cesaro}) and then to the geometry of the underlying Banach space (Section \ref{section:lpspaces}). In particular, we characterize the admissibility of sequences in Hilbert spaces by means of a (sufficient and necessary) natural criterion of ``asymptotic orthonormality'':

\begin{theoremnn}[Admissibility on a Hilbert space, Theorem \ref{thm:AdmissibilityHilbert} in Section \ref{section:lpspaces}] A sequence  $(F_n)_{n=1}^\infty$ of sets on a Hilbert space $X$ is admissible for testing compactness if and only if it is bounded in that 
$
\sup_{n} \sup_{f\in F_n}\norm{f_n}_X<+\infty
$
and asymptotically orthogonal in that for every function $f\in \bigcup_{m=1}^\infty F_m$ and for every sequence $f_n\in F_n$ we have
$
\lim_{n\to\infty} \angles{f_n,f}=0.
$
\end{theoremnn}
We also give analogous sufficiency criteria for admissibility on $L^p$ spaces (Theorem \ref{thm:criteria}). These admissibility criteria are easily verifiable and they simplify and unify various arguments in the literature used in establishing necessity of testing conditions that have been previously done case by case (Section \ref{section:relevance}).

We further illustrate applicability of the viewpoint of admissibility by characterizing the $L^p(\mu)\to L^p(\mu)$ compactness of dyadic paraproducts for general measures $\mu$  (Theorem \ref{thm:para:necessity} and Theorem \ref{thm:para:sufficiency}). This is done through introducing a space of vanishing mean oscillations (VMO) with respect to the mass of cubes (Definition \ref{definition:vmo_mass}), as opposed to the length of cubes (classical), as guided by the admissibility viewpoint. 
To further motivate the new masswise VMO space, we fully characterize (Theorem \ref{theorem:relation}) when the masswise and lengthwise VMO spaces coincide.
Finally, we prove a John--Nirenberg-type equality for the masswise VMO space (Corollary \ref{cor:JNIElast}).

\subsection*{Acknowledgements}The authors were supported by the Research Council of Finland through Projects 332740, 336323, and 358180.
\subsection*{Use of AIs}No LLMs or other AI models were used.
\section{Characterization of admissible sequences}\label{section:weaklynull}
\begin{theorem}Let $(F_i)$ be a sequence of sets in a Banach space $X$. Then the following conditions are equivalent:
\begin{itemize}
\item (Admissibility) The sequence $(F_i)$ is admissible.
\item (Weakly null) For every $f^*\in X^*$ and for every subsequence $f_{i(k)}\in F_{i(k)}$ we have
$$
\lim_{k\to \infty} \angles{f_{i(k)},f^*}= 0.
$$
\end{itemize}
\end{theorem}
\begin{proof}Note that $$\limsup_{i\to \infty} \sup_{f_i\in F_i} \norm{Tf_i}_Y=0$$
if and only if 
$$
\limsup_{k\to \infty}\norm{Tf_{i(k)}}_Y=0 
$$
for every subsequence $f_{i(k)}\in F_{i(k)}$.

{\it Admissibility implies weak-convergence.} Testing admissibility against an arbitrary continuous linear functional $f^*\in X^*$ and an arbitrary subsequence  $f_{i(k)}\in F_{i(k)}$ yields
$$
\limsup_{k\to \infty}\abs{\angles{f_{i(k)},f^*}}=0.
$$

{\it Weak-convergence implies admissibility.} Let $T:X\to Y$ be a compact linear operator from the Banach space $X$ to another Banach space $Y$, and let $f_{i_1(k)}\in F_{i_1(k)}$ be an arbitrary subsequence. By the uniform boundedness principle, the weak convergence
$$
\lim_{k\to \infty} \angles{f_{i_1(k)},f^*}= 0\quad\text{for every $f^*\in X^*$}
$$ 
implies the norm boundedness
$$
\sup_{k} \norm{f_{i_1(k)}}_X<+\infty.
$$
 Since $T$ is compact and the sequence $f_{i_1(k)}\in F_{i_1(k)}$ is bounded, there exists a further subsequence $f_{i_2(k)}$ such that $Tf_{i_2(k)}$ converges to some $g\in Y$ in norm. Now, we have
 $$
 \angles{g,g^*}=\lim_{k\to \infty} \angles{Tf_{i_2(k)},g^*}=\lim_{k\to \infty} \angles{f_{i_2(k)},T^*g^*}=0\quad\text{ for every $g^*\in Y^*$ }
 $$
and hence $g=0$. Therefore,
 $$
 \limsup_{k\to \infty}\norm{Tf_{i_1(k)}}_Y=\limsup_{k\to \infty}\norm{Tf_{i_2(k)}}_Y=\norm{g}_Y=0.
 $$
\end{proof}

\begin{lemma}[Admissibility implies boundedness]Assume that $(F_i)$ is an admissible sequence of sets in a Banach space $X$. Then
$$\sup_{i} \sup_{f_i\in F_i} \norm{f_i}_X<+\infty.$$
\end{lemma}

\begin{proof}This follows from the following two facts. First,  
$\sup_{i} \sup_{f_i\in F_i} \norm{f_i}_X<+\infty$
if and only if
$
\sup_k\norm{f_{i(k)}}_X<+\infty
$
for every subsequence $f_{i(k)} \in F_{i(k)}$. Second, every weakly convergent sequence is norm bounded by the uniform boundedness principle.\end{proof}
\section{Weak convergence versus Ces\`aro summability}\label{section:cesaro}
{\it Cancellativity} of sequences is introduced to connect the admissibility of sequences to the geometry of the underlying Banach space.

\begin{definition}[Cancellativity]
A sequence $F_i$ of sets in a Banach space $X$ is called {\it cancellative} if every subsequence $f_{i_1(k)}\in F_{i_1(k)}$ has a further subsequence $f_{i_2(k)}$ that Ces\`aro sums to zero in norm,
\begin{align}\label{eq:criteria:cancellative}
    \lim_{K\to \infty} \frac{1}{K}\Big\| \sum_{k=1}^K f_{i_2(k)}\Big\|_X=0.
\end{align}
\end{definition}

This amounts to saying that there is some asymptotic cancellation among the terms of the sequence in that the sum improves upon the (non-cancellative) triangle-inequality estimate
$$
\Big\| \sum_{k=1}^K f_{i_2(k)}\Big\|_X \leq \sup_{k=1,\ldots,K} \norm{f_{i_2(k)}}_X K.
$$
The prototypical example of a cancellative sequence is an orthonormal sequence in a  Hilbert space. Every cancellative sequence is admissible:

\begin{lemma}[Cancellativity implies admissibility]Let $X$ be a Banach space. Then every cancellative sequence is admissible.
\end{lemma}
\begin{proof}Proof by contraposition. Assume that $(F_i)$ is a non-admissible sequence of sets in a Banach space $X$. Then there exist a functional $f^*\in X^*$ and a subsequence $f_{i_1(k)}\in F_{i_1(k)}$ such that
$$
\limsup_{k\to \infty} \abs{\angles{f_{i_1(k)},f^*}}>0.
$$
Therefore, there exist a further subsequence $f_{i_2(k)}$ and a radius $r>0$ such that
$$\abs{\angles{f_{i_2(k)},f^*}}\geq r$$
for every $k=1,2,\ldots$. Thus, the complex numbers $c_k:=\angles{f_{i_2(k)},f^*}$ lie outside the ball of radius $r>0$ centered at the origin. We make the following geometric and counting observations: (1) The complex plane can be divided into four sectors such that each sector is bisected by one of the four coordinate axes. (2) The real part of the points $c_k$ in the sector bisected by the positive $x$-axis is at least $r\cos \frac{\pi}{4}$. (3) Each sector can be rotated into the sector bisected by the positive $x$-axis by multiplying by $c\in\{+1,-1,i,-i\}$. (4) At least one of the sectors has infinitely many points $c_k$. Therefore, there exists a further subsequence  $f_{i_3(k)}$ such that for some $c\in \{+1,-1,i,-1\}$ we have
$$
\mathrm{Re} \, c\cdot \angles{f_{i_3(k)},f^*}\geq r\cos \frac{\pi}{4}
$$
for all $k=1,2,\ldots$.
Now, for every further subsequence  $f_{i_4(k)}$ of the subsequence  $f_{i_3(k)}$, we have
$$
\norm{f^*}_{X^*}\norm{\sum_{k=1}^K f_{i_4(k)}}_X\geq \mathrm{Re} \, c\cdot \angles{\big(\sum_{k=1}^K f_{i_4(k)}\big),f^*} \geq K r \cos \frac{\pi}{4}.
$$
\end{proof}
While cancellativity implies admissibility, the converse may fail: whether every weakly convergent sequence has a Ces\`aro summable subsequence or not depends on the structure of the underlying Banach space $X$. A Banach space in which every weakly convergent sequence has a Ces\`aro summable subsequence is said to have the {\it  weak Banach--Saks property}. The property is named after the Banach--Saks theorem \cite{BaSa30}, 
which states that every weakly convergent sequence in $L^p$ with $p\in(1,\infty)$ has a norm Ces\`aro summable subsequence. Another example of a space with the Banach--Saks property is $L^1$, while an example of a space without the property is $\ell^\infty$. For more on the Banach--Saks property, see for example the book by Riesz and Nagy \cite{RiNa12}.

\section{Admissibility criteria in Banach spaces with geometric structure}\label{section:lpspaces}
Admissibility typically originates from the interaction between the structure of testing sequences and the geometry of the Banach space. Thus, easily verifiable criteria for admissibility are available in Banach spaces with geometric structure. 

\begin{theorem}[Admissibility in a Hilbert space]\label{thm:AdmissibilityHilbert}A sequence $F_i$ of sets in a Hilbert space $X$ is admissible if and only if it has the following two properties:
\begin{itemize}
\item (Boundedness) We have
$$
\sup_i \sup_{f_i\in F_i} \norm{f_i}_X<+\infty.
$$
\item (Asymptotic orthogonality) For every vector $f\in \bigcup_{i} F_i$ and for every sequence $f_i\in F_i$ we have
$$
\lim_{i\to \infty} \angles{f_i,f}=0.
$$
\end{itemize}
\end{theorem}
\begin{proof}The necessity of the properties is clear. The sufficiency of the properties is checked in what follows. We may assume without loss of generality that

$$
X=\overline{\Span \bigcup_{i} F_i}.
$$
Let $f_{i(k)}\in F_{i(k)}$ be a subsequence and $f^*\in X^*$ an arbitrary linear functional. By the Riesz representation theorem, there is $\bar{f}\in X$ such that
$$
\angles{x,f^*}=\angles{x,\bar{f}}\quad\text{for every $x\in X$. }
$$ 
Let $f\in X$ to be chosen later.
Now,
\begin{equation*}
\begin{split}
&\abs{\angles{f_{i(k)},f^*}}=\abs{\angles{f_{i(k)},\bar{f}}}\\
&\leq \abs{\angles{f_{i(k)},f-\bar{f}}}+ \abs{\angles{f_{i(k)},f}}\leq \big(\sup_i \sup_{f_i\in F_i} \norm{f_i}_X\big) \norm{f-\bar{f}}_X+\abs{\angles{f_{i(k)},f}}.
\end{split}
\end{equation*}
By the assumption on boundedness and by density, we can choose $f\in \Span\bigcup_i F_i$ so that
$$
\big(\sup_i \sup_{f_i\in F_i} \norm{f_i}_X\big) \norm{f-\bar{f}}_X\leq \varepsilon.
$$
Further, by the assumption on inner products, we can choose $K$ such that
$\abs{\angles{f_{i(k)},f}}\leq \varepsilon$
for every $k\geq K$.
\end{proof}
The next admissibility criteria is in Banach function spaces on which the triangle inequality can be improved for disjointly supported functions.

\begin{theorem}[Criteria for admissible sequences on Banach function spaces]\label{thm:criteria} Let $X$ be a Banach function space equipped with a function $\phi:\bn_+\to [0,+\infty)$ that improves upon triangle inequality as follows: 
$$
\norm{\sup_{n=1,\ldots,N}f_n}_X\leq \phi(N),\qquad \lim_{K\to \infty} \frac{\phi(K)}{K}=0,
$$
whenever $\{f_n\}_{n=1}^N$ are pairwise disjointly supported, $\max_{n=1,\ldots,N}\norm{f_N}_X\leq 1$, and $N\in \bn$. 
Let $F_i\subset X$ be a sequence of sets such that 
$$
\sup_i \sup_{f_i\in F_i} \norm{f_i}_X<+\infty.
$$
Then sufficient for the admissibility of the sequence $F_i$
is either of the following conditions:
\begin{itemize}
\item For each fixed $f\in \bigcup_i F_i$ and for every sequence $f_i\in F_i$ we have
$$\lim_{i\to\infty}   \|1_{\supp f}f_{i}\|_X=0.
$$
\item For each fixed $f\in \bigcup_i F_i$ and for every sequence $f_i\in F_i$ we have
$$\lim_{i\to\infty} \| 1_{\supp(f_i)} f\|_X=0.
$$
\end{itemize}
\end{theorem}
\begin{proof}[Proof of Theorem \ref{thm:criteria}]We prove that the sequence $F_i$ of sets is cancellative. Let $f_{i_1(k)}\in F_{i_1(k)}$ be an arbitrary subsequence. Depending on which of the above bulleted conditions holds, we choose a further subsequence, now denoted by $f_i$ for brevity, such that for all $i=1,2,\ldots$ the same one of the following two alternatives holds:
\begin{align*}
	\max_{k\leq i-1} \| 1_{\supp(f_k)}f_{i}\|_X \leq 2^{-2i},\qquad \max_{k\leq i-1} \| 1_{\supp(f_i)} f_k\|_X \leq 2^{-2i}.
\end{align*}
If the left alternative holds, denote $D(i) := \bigcup_{k=1}^{i}\supp(f_i),$ and for $i\geq 2$ we bound 
\begin{align*}
	  \|1_{D(i-1)}f_i\|_X \leq  \sum_{k=1}^{i-1}\| 1_{\supp(f_k)}f_{i}\|_X 
	\leq \sum_{k=1}^{i-1} 2^{-2i} \leq 2^{-i}.
\end{align*}
Summing a geometric series and the improved triangle inequality of $X$ applied with the uniformly bounded and disjointly supported functions $1_{\supp(f_i)\setminus D(i-1)}f_i$ gives
\begin{align*}
	\big\| \sum_{i=1}^N f_i\big\|_X \leq 	\sum_{i=2}^N \| 1_{D(i-1)}f_i \|_X + 		\big\| \sum_{i=1}^N 1_{\supp(f_i)\setminus D(i-1)}f_i\big\|_X \leq 1+\phi(N).
\end{align*}
If the right alternative holds, denote $U(i) := \bigcup_{k=i}^{\infty}\supp(f_k),$ and then we bound
\begin{align*}
	\| 1_{U(i+1)}f_i \|_X \leq  \sum_{k=i+1}^{\infty}\| 1_{\supp(f_k)}f_i\|_X  \leq \sum_{k=i+1}^{\infty} 2^{-2k} \leq 2^{-i}.
\end{align*}
Summing a geometric series and the improved triangle inequality of $X$ applied with the uniformly bounded and disjointly supported functions $1_{\supp(f_i)\setminus U(i+1)}f_i$ gives 
\begin{align*}
	\big\| \sum_{i=1}^N f_i\big\|_X \leq 	\sum_{i=1}^N \| 1_{U(i+1)} f_i  \|_X + 		\big\| \sum_{i=1}^N 1_{\supp(f_i)\setminus U(i+1)}f_i\big\|_X  \leq 1+\phi(N).
\end{align*}
In either case, we have thus shown that 
\begin{align*}
     \frac{1}{N}\big\| \sum_{i=1}^N f_i\big\|_X \leq \frac{1+\phi(N)}{N}.
\end{align*}
Taking the limit $N\to \infty$ concludes the proof of the cancellativity.
\end{proof}

\section{Relevance for necessary conditions in $T(1)$ theorems for compactness}\label{section:relevance}
We illustrate the applicability of the theory in establishing the necessity of testing conditions for compactness by connecting the theory to a few examples from literature. For the purposes of the exposition, we introduce an ad hoc terminology of {\it control estimates}:

 Let $X$ be a Banach function space on a measure space $(\Omega,\mu)$ and $Y$ a Banach space. Let $T:X\to Y$ be a linear operator. Assume that $\{(c_{T,Q},f_{T,Q})\}_{Q\in\cq_T}$ is family of pairs of coefficients $c_Q\in [0,+\infty)$ and functions $f_Q\in X$ indexed by a collection $\cq_T$ of measurable sets on $\Omega$ such that the following {\it control estimate} holds:
$$
c_{T,Q}\lesssim \norm{Tf_{T,Q}}_Y
\quad\text{for all $Q\in\cq_T$}.$$  
The {\it control sets, coefficients, and functions} depend on the structure of the specific operator $T:X\to Y$. 
A few well-known examples from literature of such estimates include:
\begin{itemize}
\item (Mean oscillations and dyadic paraproducts on $L^p$) We have
$$
\frac{1}{\abs{Q}}\int_Q \abs{b-\angles{b}_Q}\dx\lesssim \Big\|P_b \Big(\frac{1_Q}{\norm{1_Q}_{L^p}}\Big)\Big\|_{L^p}\quad\text{for all dyadic cubes $Q$}.
$$ 
For the $L^p(\lambda)\to L^q(\mu)$ estimates with $p\leq q$, $\lambda\in A_p$, and $\mu\in A_q$ there are analogous estimates.

\item (Weak-boundedness property and Calder\'on--Zygmund operators on $L^2$) We have
$$
\frac{1}{\abs{Q}}\big|\angles{1_Q,T1_Q}\big|\leq \Big\|T\Big(\frac{1_Q}{\norm{1_Q}_{L^2}}\Big)\Big\|_{L^2}\quad\text{for all cubes $Q$.}
$$
\end{itemize}
As well-known, the conditions in the $T(1)$ theorems for boundedness are typically of the form
$$
\sup_{Q\in\cq_T} c_{T,Q}<+\infty,
$$
while 
 the conditions in the $T(1)$ theorems for compactness are typically of the form
$$
\lim_{i\to \infty}\sup_{Q\in\cq_{T,i}} c_{T,Q}=0
$$
for some decreasing sequence (or several thereof) $\cq_T\supseteq \cq_{T,1}\supseteq \cq_{T,2}\ldots$ depending on the specific context. For example, in the context of the $L^p(\dx)\to L^q(\dx)$ with $p\leq q$ compactness of commutators (see \cite{GHWY21}) and paraproducts, the decreasing sequences of cubes are as follows:
\begin{itemize}
\item (Large cubes) $
\cq_i^{\mathrm{large}}:=\{Q : \ell(Q)\geq i\}$. 
\item (Small cubes) $
\cq_i^{\mathrm{small}}:=\{Q : \ell(Q)\leq \frac{1}{i}\}.
$
\item (Distant cubes) $
\cq_i^{\mathrm{distant}}:=\{Q : \dist(0,Q)\geq i\}.
$ 
\end{itemize}
The condition
$
\lim_{i\to \infty}\sup_{Q\in\cq_{T,i}} c_{T,Q}=0
$
is often established via the estimate
$
c_{T,Q}\lesssim \norm{Tf_{T,Q}}_Y
$ by proving that
$$
\lim_{i\to \infty}\sup_{Q\in\cq_{T,i}} \norm{T{f_{T,Q}}}_Y=0.
$$
This is typically proven by constructing an auxiliary sequence $f_{Q_i}$ with $Q_i\in\cq_i$ that satisfies the following convergence criterion (or a variant thereof). The criterion was originally implicit in Uchiyama \cite{Uch1978} and other later works, in the following form it was stated in \cite{HOS2023}.
\begin{theorem}[Uchiyama's convergence criterion, \cite{Uch1978, HOS2023}] Let $1<p,q<+\infty$ and $\lambda,\mu$ be weights and let $T:L^p(\lambda)\to L^q(\mu)$ be a bounded linear operator. Let $f_i$ be a sequence in $L^p(\lambda)$ that satisfies the following criteria:
\begin{itemize}
\item (Boundedness) We have
$$
\sup_{i} \norm{f_i}_{L^p(\mu)}<+\infty.
$$
\item (Pairwise disjointness) The functions $f_i$ are pairwise disjointly supported.
\end{itemize}
Then, if $Tf_i$ converges to $g\in L^q(\mu)$ in norm, then 
$
g=0.
$
\end{theorem}
\begin{remark}[Generalization of Uchiyama's convergence criterion]The Uchiyama convergence criterion can be generalized as follows: Let $T:X\to Y$ be a bounded operator from a normed space $X$ to another normed space $Y$. Then the cancellation assumption $$
\limsup_{N\to \infty}\frac{1}{N}\Big\|\sum_{n=1}^N f_n\Big\|_{X}=0 
$$
ensures that if $Tf_i$ converges to $y$ in $Y$, then $y=0$. Indeed, since every convergent sequence is Ces\`aro summable, we obtain
\begin{align*}
    \norm{y}_Y=\lim_{N\to \infty} \Big\|\frac{1}{N}\big(\sum_{n=1}^N Tf_n\big)\Big\|_Y &=\lim_{N\to \infty} \frac{1}{N} \Big\|T\big(\sum_{n=1}^N f_n\big)\Big\|_Y \\ &\leq \norm{T}_{X\to Y} \lim_{N\to \infty} \frac{1}{N} \norm{\sum_{n=1}^N f_n}_X=0.
\end{align*}
\end{remark}

The results in this paper characterize what sequences are admissible for testing compactness (independent of the structure of the operator) and gives easily verifiable, rather general admissibility criteria (under rather weak assumptions on the geometry of the underlying Banach function space). Thereby, the paper contributes to the study of the conditions of the form
$$
\lim_{i\to \infty}\sup_{Q\in\cq_{T,i}} \norm{T{f_{T,Q}}}_Y=0\quad\text{for $\cq_T\supseteq \cq_{T,1}\supseteq \cq_{T,2}\ldots$}
$$
in two interrelated ways. First, the abstract viewpoint helps to identify sequences $\cq_T\supseteq \cq_{T,1}\supseteq \cq_{T,2}\ldots$ of sets that are relevant in a specific context. Second, the admissibility criteria for the sequence $\{f_{T,Q}\}_{Q\in \cq_{T,i}}$ help to establish the condition $$
\lim_{i\to \infty}\sup_{Q\in\cq_{T,i}} \norm{T{f_{T,Q}}}_Y=0.
$$
This is illustrated in the next section by characterizing the compactness of dyadic paraproducts on $L^p$ by means of vanishing mean oscillations on general measure spaces.
\begin{remark}
Admissibility criteria are satisfied by sequences $\{f_{T,Q}\}_{Q\in \cq_{T,i}}$ appearing in testing conditions for compactness in various contexts, such as the $T(1)$ theorem for compactness of Calder\'on--Zygmund operators (see e.g. \cite{Vil2015, MitSto23, FrGrWi24}), the VMO characterization of compactness of commutators (see e.g. \cite{Uch1978, GHWY21, HOS2023}) and of paraproducts (see \cite{ChaoLi96}). The verification of an admissibility criterion for a sequence is often straightforward - for example:
\begin{itemize}
\item Context: The characterization of the $L^p(\dx)$ compactness of dyadic paraproducts by means of vanishing mean oscillations along large, small, and distant cubes.
\item Sets: The collection $\cq$ of dyadic cubes.
\item Subsets: $\cq_i^{\mathrm{large}}, \cq_i^{\mathrm{small}},$ and $\cq_i^{\mathrm{distant}}$.
\item Functions: $\{f_Q\}:= \{\frac{1_Q}{\abs{Q}^{1/p}}\}$.

\item The sequence $\{f_Q\}_{Q\in \cq_i^{\mathrm{large}}}$ satisfies the admissibility criterion
$$
\lim_{\ell(Q)\to \infty} \Big\|1_R \frac{1_Q}{\abs{Q}^{1/p}}\Big\|_{L^p}=\lim_{\ell(Q)\to \infty} \Big(\frac{\abs{Q\cap R}}{\abs{Q}}\Big)^{1/p}=0.
$$
\item  Similarly for the sequences associated with small and distant cubes.
\end{itemize}

 \end{remark}

\section{Compactness of dyadic paraproducts on general measure spaces}\label{section:paraproducts}
\subsection{Classical results on boundedness and compactness}
In this section we characterize the compactness of paraproducts for general locally finite Borel measures. We begin by recalling some classical results around the boundedness of paraproducts. 
First of all, the mean oscillation of a locally integrable function $b$ over a set $Q$ is defined by 
$$
\| b - \langle b\rangle_Q\|_{\avL^p(Q)}:=\Big(\frac{1}{\mu(Q)}\int_Q \abs{b-\angles{b}_Q}^p \dmu\Big)^{1/p}.
$$
For $\mu(Q)=0$, the convention $\| b - \langle b\rangle_Q\|_{\avL^p(Q)}:=0$ is used.
This convenient notation for $L^p(Q)$-averages was introduced in \cite{NPTV17}.
The dyadic $\bmo^p(\mathbb{R}^d,\mathcal{D},\mu)$ (bounded mean oscillation) norm of a function $b\in L^1_{\mathrm{loc}}(\mu)$ is defined by
$$
\norm{b}_{\bmo^p(\mathbb{R}^d,\mathcal{D},\mu)}:=\sup_{Q\in\mathcal{D}} \lpq{b-\angles{b}_Q}.
$$
The dyadic paraproduct $P_{b,\mathcal{D}}$ associated with a locally integrable function $b$ is defined by 
\begin{align*}
    P_{b,\mathcal{D}}f := \sum_{Q\in\mathcal{D}} D_Qb \langle f \rangle_Q1_Q,\qquad D_Qb := \sum_{P\in\operatorname{ch}_\cd (Q)}(\langle b\rangle_P - \langle b\rangle_Q)1_P.
\end{align*}
Above the collection $\operatorname{ch}_\cd(Q)$ of the dyadic children of a cube $Q$ is defined to be the collection of all the maximal (with respect to set inclusion) cubes $P\in\mathcal{D}$ such that $P\subsetneq Q$.

Dyadic paraproducts and mean oscillations are closely connected. As is well-known, mean oscillations are controlled by testing the paraproduct,
\begin{align}\label{eq:bmoLWpara}
    \lpq{b-\angles{b}_Q}\leq 2\norm{P_b\frac{1_Q}{\mu(Q)^{1/p}}}_{L^p(\mathbb{R}^d,\mu)}.
\end{align}
To see this, note that by linearity of averages and Jensen's inequality,
\begin{align*}
    \lpq{b-\angles{b}_Q} = \inf_c\lpq{b-c+\angles{b-c}_Q}\leq 2\inf_c\lpq{b-c},
\end{align*}
and therefore in particular
\begin{align*}
    \lpq{b-\angles{b}_Q} &\leq 2 \lpq{b-\angles{b}_Q+\sum_{R\supsetneq Q} D_Rb \angles{1_Q}_R}\\
&=2 \lpq{\sum_{R\subseteq Q}D_Rb \angles{1_Q}_R +\sum_{R\supsetneq Q} D_Rb \angles{1_Q}_R} \\
&\leq 2\norm{P_b\frac{1_Q}{\mu(Q)^{1/p}}}_{L^p(\mathbb{R}^d,\mu)}.
\end{align*}  
By \eqref{eq:bmoLWpara}, the dyadic $\bmo$ norm is bounded by the operator norm of the dyadic paraproduct,
$$
\norm{b}_{\bmo^p(\mathbb{R}^d,\mathcal{D},\mu)} \lesssim \norm{P_{b,\mathcal{D}}}_{L^p(\mathbb{R}^d,\mu)\to L^p(\mathbb{R}^d,\mu)}.
$$
Classical results (see e.g. \cite{Meyer90Oondelettes} and \cite{ChaoLo92}) establish that also the reverse holds and hence these norms are in fact comparable:
\begin{theorem}[Characterization of the boundedness of dyadic paraproducts, see e.g. \cite{Meyer90Oondelettes} and \cite{ChaoLo92}] Let $\mathcal{D}$ be a dyadic system and $\mu$ a locally finite Borel measure on $\br^d$. Then,
$$
\norm{P_{b,\mathcal{D}}}_{L^p(\mathbb{R}^d,\mu)\to L^p(\mathbb{R}^d,\mu)}\sim_p \norm{b}_{\bmo^p(\mathbb{R}^d,\mathcal{D},\mu)}.
$$
\end{theorem}
For a short proof of this comparison, see for example Remark 4 in \cite{hanninenhytonen2016}.

When the underlying measure is the Lebesgue-measure $\mathcal{L}$ on the unit interval $[0,1]$, it was remarked in \cite{ChaoLi96} that the \emph{compactness} of dyadic paraproducts is characterized by the condition that the mean oscillations along small dyadic cubes vanish,
$$
\lim_{M\to \infty} \sup_{\substack{Q\in\mathcal{D}:\\\ell(Q)\leq 1/M}} \osc=0
$$Since the unit interval is bounded, the large and far cubes do not appear. However, it is immediate that the proof in \cite{ChaoLi96} works also in the space $\mathbb{R}^d$, now requiring that the mean oscillations vanish also along large and far cubes,
$$
\lim_{M\to \infty} \sup_{\substack{Q\in\mathcal{D}:\\ \ell(Q)\geq M}} \osc=0
$$
and
$$
\lim_{M\to \infty} \sup_{\substack{Q\in\mathcal{D}:\\ \dist(0,Q) \geq M}} \osc=0.
$$
As usual, it is said that $b\in \vmo^p(\mathbb{R}^d,\mathcal{D},\mathcal{L})$ if $b$ satisfies the above three conditions.

\begin{theorem}[\cite{ChaoLi96}] Let $p\in(1,\infty).$ Then $P_{b,\mathcal{D}}:L^p(\mathbb{R}^d,\mathcal{L})\to L^p(\mathbb{R}^d,\mathcal{L})$ is compact if and only if $b\in\vmo^p(\mathbb{R}^d,\mathcal{D},\mathcal{L})$.
\end{theorem}

\subsection{Vanishing mean oscillations on general measure spaces}
The metric concepts of side-length and of distance (which appear in the above definition of the $\vmo$ space) are absent on general (non-metric) measure spaces. Even for a general Borel measure on the metric space $\br^d$, the measure and the side-length of cubes are in general incomparable (e.g. point masses and zero measures). The vanishing mean oscillations compatible with general measure spaces are introduced as follows:

\begin{definition}[Vanishing mean oscillations for general measures]\label{defn:VMOmu} Let $\mu$ be a locally finite Borel measure on $\br^d$. A function $b\in\bmo^p(\mathbb{R}^d,\mathcal{D},\mu)$
 belongs to the dyadic $\vmo^p(\br^d,\mathcal{D},\mu)$ space if its mean oscillations vanish along heavy, light, and distant cubes:
\begin{itemize}
\item (Heavy cubes)
$$
\lim_{M\to \infty} \sup_{\substack{Q\in\mathcal{D}:\\ \mu(Q)\geq M}} \osc=0,
$$
\item (Light cubes)
$$
\lim_{M\to \infty} \sup_{\substack{Q\in\mathcal{D}:\\\mu(Q)\leq 1/M}} \osc=0,
$$
\item (Distant cubes) 
$$
\lim_{M\to \infty} \sup_{\substack{Q\in\mathcal{D}:\\ \dist(0,Q) \geq M}} \osc=0.
$$
\end{itemize}
\end{definition}
On the Euclidean space equipped with a general Borel measure, a $\vmo$ space can be defined either using lengths or measures. The relation between these two versions of the VMO space is studied in Section \ref{section:relation}, culminating in the full characterization stated as Theorem \ref{theorem:relation}. In particular, regarding the compactness of dyadic paraproducts on $L^p(\mu)$ for a general measure $\mu$, the measure version of a VMO space turns out to be both necessary and sufficient, whereas the length version of a VMO space fails to be necessary in general.

\begin{remark}[A priori assumptions on the $\vmo$-symbols] We have defined the space $\vmo^p(\mathbb{R},\mathcal{D},\mu)$ as a subspace of $\bmo^p(\mathbb{R}^d,\mathcal{D},\mu)$. On the level of Theorems \ref{thm:para:necessity} and \ref{thm:para:sufficiency} this definition indeed exactly captures when the dyadic paraproduct is compact, without any extra background assumptions. However, the a priori weaker assumption $b\in L^1_{\operatorname{loc}}(\mathbb{R}^d,\mu)$ would suffice, as it together with the three vanishing conditions (heavy, light and distant cubes) implies that in fact $b\in\bmo^p(\mathbb{R}^d,\mathcal{D},\mu).$ This implication follows from splitting the collection of all cubes into the same collections as is done in the proof of Theorem \ref{thm:para:sufficiency} below. The details are left to the interested reader. 
\end{remark}

Whereas the notions of heavy and light cubes are defined by means of measure, the notion of distant cubes is still defined by means of metric in Definition \ref{defn:VMOmu}. An equivalent, solely measure-theoretical substitute for distant cubes is provided in the following remark. However,  for simplicity of exposition, we work mostly with the hybrid measure-theoretic-metric definition of Definition \ref{defn:VMOmu}.
\begin{remark}[Solely measure-theoretical definition for $\vmo$] Let $\Omega$ be a set. Let $(\mathcal{D}_k)_{k=-\infty}^{+\infty}$ be a sequence of refining partitions by countably many sets. Set $\mathcal{D}:=\bigcup_{k=-\infty}^\infty \mathcal{D}_k$ and $\cf:=\sigma(\mathcal{D})$. Let $\mu:\cf\to[0,+\infty]$ be a  measure. Assume that $\mu$ is locally finite in the sense that $\mu(Q)<+\infty$ for every $Q\in\mathcal{D}$. In this more general setting, the generalization of the vanishing mean oscillation along distant cubes is the condition
$$
\lim_{M\to \infty} \sup_{\substack{Q\in\mathcal{D}:\\Q\subseteq \Omega_M^c}} \osc=0,
$$
where $(\Omega_M)_{M=1}^\infty$ is some fixed sequence of sets such that $\Omega_M\uparrow \Omega$ and such that each $\Omega_M$ is a finite union of sets in $\mathcal{D}$. The corresponding space $\vmo^p(\Omega,\mathcal{D},\mu)$ is independent of the specific choice of such sequence $\Omega_M$. Moreover, the metric definition on $\br^d$ is recovered through any sequence $\Omega_M$ with
$B(0,M-1) \subseteq \Omega_M \subseteq B(0,M)$.
\end{remark}

\subsection{Necessity}
The necessity of $b\in\vmo^p(\mathbb{R}^d,\mathcal{D},\mu)$ for the $L^p(\br^d,\mu)$ compactness of the dyadic paraproduct $P_{\mathcal{D},b}$ is immediate by the generic machinery. 
\begin{theorem}[Necessity]\label{thm:para:necessity}Let $\mu$ be a locally finite Borel measure on $\br^d$ and $b\in L^1_{\operatorname{loc}}(\mathbb{R}^d,\mu)$. If $P_{b,\mathcal{D}}:L^p(\mathbb{R}^d,\mu)\to L^p(\mathbb{R}^d,\mu)$ is compact, then
$
b\in \vmo^p(\mathbb{R}^d,\mathcal{D},\mu).
$
\end{theorem}
\begin{proof}
As compact operators are bounded, the inclusion $b\in\bmo^p(\mathbb{R}^d,\mathcal{D},\mu)$ follows immediately from \eqref{eq:bmoLWpara}, stating that for all $Q\in\mathcal{D}:$
\begin{align}\label{eq:bmoLWparaX}
\lpq{b-\angles{b}_Q}\leq 2\Big\| P_b\frac{1_Q}{\mu(Q)^{1/p}}  \Big\|_{L^p(\mathbb{R}^d,\mu)}.
\end{align}
To check the vanishing conditions, we note that
the space $L^p(\mathbb{R}^d,\mu)$ has geometric structure as in Theorem \ref{thm:criteria} for checking admissibility of sequences, because
$$
\norm{\sup_{n=1,\ldots,N} f_n}_{L^p(\mathbb{R}^d,\mu)}\leq \sup_{n=1,\ldots,N} \norm{f_n}_{L^p(\mathbb{R}^d,\mu)}  N^{1/p}.
$$ 
For a given cube $Q\in\mathcal{D}$ let us denote 
\begin{align*}
    f_Q := 
    \begin{cases}
        1_Q/\mu(Q)^{1/p},\quad &\mu(Q)>0 \\
        0,\quad&\mbox{else.}
    \end{cases}
\end{align*}
Provided $f_Q\not = 0,$ there holds that 
\begin{align*}
    \| 1_{\supp(f_R)} f_Q \|_{L^p(\mathbb{R}^d,\mu)} \leq \Big\|1_R\frac{1_Q}{\mu(Q)^{1/p}}\Big\|_{L^p(\mathbb{R}^d,\mu)} = \Big(\frac{\mu(R\cap Q)}{\mu(Q)}\Big)^{1/p}.
\end{align*}
Now, we check that the following sequences 
meet an admissibility criteria. 
\begin{itemize}
\item The sequence $\{f_Q\not = 0 : \mu(Q)\geq M\}$ of sets of functions meets the criterion
\begin{align*}
    \lim_{\mu(Q)\to \infty}  \| 1_{\supp(f_R)} f_Q \|_{L^p(\mathbb{R}^d,\mu)}  \leq \lim_{\mu(Q)\to \infty}\Big(\frac{\mu(R\cap Q)}{\mu(Q)}\Big)^{1/p}=0.
\end{align*}
\item The sequence $\{f_Q\not = 0 : \mu(Q)\leq M^{-1}\}$ of sets of functions meets the criterion
$$
\lim_{\mu(R)\to 0}  \| 1_{\supp(f_R)} f_Q \|_{L^p(\mathbb{R}^d,\mu)} \leq \lim_{\mu(R)\to 0}  \Big(\frac{\mu(R\cap Q)}{\mu(Q)}\Big)^{1/p}=0.
$$
\item The sequence $\{f_Q\not = 0:  \dist(0,Q)\geq M\}$ of sets of functions meets the criterion
$$
\lim_{\dist(R,0)\to \infty}  \| 1_{\supp(f_R)} f_Q \|_{L^p(\mathbb{R}^d,\mu)}  \leq \lim_{\dist(R,0)\to \infty} \Big(\frac{\mu(R\cap Q)}{\mu(Q)}\Big)^{1/p}= 0.
$$
\end{itemize}
The proof is completed by \eqref{eq:bmoLWparaX}.
\end{proof}

\subsection{Sufficiency}
In this section we prove the converse to Theorem \ref{thm:para:necessity}
\begin{theorem}[Sufficiency]\label{thm:para:sufficiency}Let $\mu$ be a locally finite Borel measure on $\br^d.$ If $b\in \vmo^p(\mathbb{R}^d,\mathcal{D},\mu),$ then $P_{b,\mathcal{D}}:L^p(\mathbb{R}^d,\mu)\to L^p(\mathbb{R}^d,\mu)$ is compact.
\end{theorem}
The proof of Theorem \ref{thm:para:sufficiency} proceeds by reducing the initial collection of dyadic cubes finitely many times (with the number of reductions bounded by $C$), at each step paying a price of an $\varepsilon$ in the operator norm, until the operator corresponding to the remaining collection is of some finite rank $K(\varepsilon)$ and hence compact. The rank $K(\varepsilon)$ may blow up as $\varepsilon$ tends to zero, but the bound $C$ for the number of reductions is absolute, not depending on $\varepsilon.$  Each reduction step is formalized in the following lemma:
\begin{lemma}[Reduction step]\label{lem:reduce}Let $\mathcal{Q}$ be a collection of dyadic cubes. Let $X$ and $Y$ be Banach function spaces. Assume that for every $\varepsilon>0$ there is a subcollection $\ce(\varepsilon)\subseteq \mathcal{Q}$ such that in the corresponding decomposition
$$
P_{\mathcal{Q}}:=\sum_{Q\in\mathcal{Q}} P_Q=\sum_{Q\in \mathcal{Q}\setminus\ce(\varepsilon)}P_Q+\sum_{Q\in\ce(\varepsilon)}P_Q=:P_{\mathcal{Q}\setminus\ce(\varepsilon)}+P_{\ce(\varepsilon)}
$$
the operator $P_{\mathcal{Q}\setminus\ce(\varepsilon)}:X\to Y$ is compact and $\norm{P_{\ce(\varepsilon)}}_{X\to Y}\leq \varepsilon.$ Then, the operator $P_\mathcal{Q}:X\to Y$ is compact.
\end{lemma}
\begin{proof}The space of compact operators between Banach spaces is closed in the operator norm topology.
\end{proof}
 The operator norms of the discarded subcollections are controlled by the vanishing mean oscillation assumptions together with the following lemmata.
\begin{lemma}[Operator norm over any collection of cubes, see e.g. Remark 4 in \cite{hanninenhytonen2016}]\label{lem:OpNorm:Carleson} Let $\mathcal{Q}$ be any collection of dyadic cubes. Then,
$$
\norm{P_{b,\mathcal{Q}}}_{L^p(\mathbb{R}^d,\mu)\to L^p(\mathbb{R}^d,\mu)}\sim_ p \sup_{Q\in \mathcal{Q}} \frac{1}{\mu(Q)^{1/p}}\Big\| \sum_{\substack{R\in\mathcal{Q}:\\ R\subseteq Q}} D_Rb \Big\|_{L^p(\mathbb{R}^d,\mu)}. 
$$
\end{lemma}

\begin{definition}A collection $\mathcal{Q}\subseteq \mathcal{D}$ is called {\it connected} if for every pair $Q,R\in \mathcal{Q}$ the following holds: if $S\in\mathcal{D}$ with $Q\subseteq S\subseteq R$, then $S\in\mathcal{Q}$.
\end{definition}

\begin{lemma}[Operator norm over a connected collection of cubes]\label{lem:OpNorm:connected}Let $\mathcal{Q}$ be a connected collection of dyadic cubes. 
Then,
$$
\Big\| \sum_{\substack{R\in\mathcal{Q}:\\ R\subseteq Q}} D_Rb \Big\|_{L^p(\mathbb{R}^d,\mu)} \lesssim  \Big\| \sup_{  \substack{ R,S\in \mathcal{Q} \\ R\subseteq S\subseteq Q}} \abs{\angles{b}_R-\angles{b}_S} 1_S\Big\|_{L^p(\mathbb{R}^d,\mu)}.
$$
\end{lemma}
\begin{proof}
To see how the claim follows from Lemma \ref{lem:OpNorm:Carleson}, denote
\begin{align*}
    R_*(x,L,Q)&:=\min \{ R\in \mathcal{Q} \cap \mathcal{D}_{[-L,L]} : x\in  R\subseteq Q\}, \\
R^*(x,L,Q)&:=\max \{ R\in \mathcal{Q} \cap \mathcal{D}_{[-L,L]} : x\in R\subseteq Q\},
\end{align*}
with the minima and maxima of set inclusion.
By connectedness of $\mathcal{Q},$ the martingale differences
$$
D_Qb := \sum_{P\in\operatorname{ch}_\cd (Q)}(\langle b\rangle_P - \langle b\rangle_Q)1_P
$$telescope and hence
\begin{equation*}
\begin{split}
&\sum_{\substack{R\in\mathcal{Q}:\\ R\subseteq Q}} D_Rb(x)=\lim_{L\to \infty}  \left(\angles{b}_{R_*(x,L,Q)}-\angles{b}_{R^*(x,L,Q)}\right).
\end{split}
\end{equation*}
Now simply observe that 
$$
\abs{\lim_{L\to \infty}  \left(\angles{b}_{R_*(x,L,Q)}-\angles{b}_{R^*(x,L,Q)}\right)}\leq 
\sup_{  \substack{ R,S\in \mathcal{Q} \\ R\subseteq S\subseteq Q}} \abs{\angles{b}_R-\angles{b}_S} 1_S.
$$
\end{proof}

\subsubsection*{Proof of Theorem \ref{thm:para:sufficiency} with reductive steps}
\subsubsection*{Step 1} By Lemma \ref{lem:OpNorm:BoundedMeasure} the operator $P_{b,\mathcal{D}}:L^p(\mathbb{R}^d,\mu)\to L^p(\mathbb{R}^d,\mu)$ is compact if the operator $P_{b,\mathcal{Q}_1}:L^p(\mathbb{R}^d,\mu)\to L^p(\mathbb{R}^d,\mu)$ is compact for every $\mathcal{Q}_1\subseteq \mathcal{D}$ with cubes of measure
\begin{itemize}
\item (Bounded measures) $\mu(Q)\sim 1$ for every $Q\in\mathcal{Q}_1$.
\end{itemize}

\subsubsection*{Step 2} By Lemma \ref{lem:OpNorm:FiniteMeasureSpace} the operator $P_{b,\mathcal{Q}_1}:L^p(\mathbb{R}^d,\mu)\to L^p(\mathbb{R}^d,\mu)$ is compact if the operator $P_{b,\mathcal{Q}_2}:L^p(\mathbb{R}^d,\mu)\to L^p(\mathbb{R}^d,\mu)$ is compact for every further subcollection $\mathcal{Q}_2\subseteq \mathcal{Q}_1$ that can be written as $\mathcal{Q}_2 = \mathcal{Q}_2^{\operatorname{in}} \cup \mathcal{Q}_2^{\operatorname{out}},$ where the outer cubes $\mathcal{Q}_2^{\operatorname{out}}$ satisfy:
\begin{itemize}
\item (Coarse scales)  $\mathcal{Q}_2^{\operatorname{out}}\subseteq\mathcal{Q}_1\cap  \mathcal{D}_{\leq L_0}$ for some $L_0\in \bz,$ and 
\item (Chaining condition) $\mathcal{Q}_2^{\operatorname{out}}:=\{Q_1\subseteq Q_2\subseteq \ldots\};$
\end{itemize}
and the inner cubes $\mathcal{Q}_2^{\operatorname{in}}$ satisfy:
\begin{itemize}
\item (Fine scales)  $\mathcal{Q}_2^{\operatorname{in}}\subseteq \mathcal{Q}_1\cap \mathcal{D}_{\geq L_0}$ for some $L_0\in \bz,$ and
\item (Finite support) $\mu(\bigcup_{Q\in \mathcal{Q}_2^{\operatorname{in}}} Q)<+\infty$.
\end{itemize}

\subsubsection*{Step 3} The operator $P_{b,\mathcal{Q}_2^{\operatorname{out}}} :L^p(\mathbb{R}^d,\mu)\to L^p(\mathbb{R}^d,\mu)$ is compact by Lemma \ref{lem:OpNorm:OuterCubes}, and the operator $P_{b,\mathcal{Q}_2^{\operatorname{in}}}:L^p(\mathbb{R}^d,\mu)\to L^p(\mathbb{R}^d,\mu)$ is compact by Lemma \ref{lem:OpNorm:InnerCubes}.

\subsubsection*{The lemmata used in the reductive steps}

\begin{lemma}[Reduction to cubes of bounded measure]\label{lem:OpNorm:BoundedMeasure} Let $\mathcal{Q}$ be a collection of dyadic cubes. Assume the vanishing mean oscillations along heavy and light cubes:
\begin{itemize}
\item (Cubes with large measure) 
$$
\lim_{M\to \infty} \sup_{\substack{Q\in\mathcal{Q}:\\ \mu(Q)\geq M}} \osc=0,
$$
\item (Cubes with small measure) 
$$
\lim_{M\to \infty} \sup_{\substack{Q\in\mathcal{Q}:\\\mu(Q)\leq M^{-1}}} \osc=0.$$
\end{itemize}
Then, the operator $P_{b,\mathcal{Q}}:L^p(\mathbb{R}^d,\mu)\to L^p(\mathbb{R}^d,\mu)$ is compact if the operator $P_{b,\mathcal{R}}:L^p(\mathbb{R}^d,\mu)\to L^p(\mathbb{R}^d,\mu)$ is compact for every subcollection $\mathcal{R}\subseteq \mathcal{Q}$ such that
\begin{itemize}
\item (Bounded measures) $\mu(R)\sim 1$ for every $R\in\mathcal{R}$.
\end{itemize}
\end{lemma}
\begin{proof}[Proof of Lemma \ref{lem:OpNorm:BoundedMeasure}] Denote 
$$
\mathcal{Q}^{\operatorname{avg}}(M): = \Big\{ Q \in \mathcal{Q}: M^{-1} \leq \mu(Q) \leq M \Big\}
$$
and split 
$$
P_{b,\mathcal{Q}} = P_{b,\mathcal{Q}^{\operatorname{avg}}(M)}  + P_{b,\mathcal{Q}\setminus \mathcal{Q}^{\operatorname{avg}}(M)}.
$$
Directly by assumption $P_{b,\mathcal{Q}^{\operatorname{avg}}(M)}$ is compact and by the vanishing oscillations along the light and the heavy cubes, the operator norm of $P_{b,\mathcal{Q}\setminus \mathcal{Q}^{\operatorname{avg}}(M)}$ can be made arbitrarily small by choosing $M$ sufficiently large.  
\end{proof}

\begin{lemma}[Reduction to finite measure space]\label{lem:OpNorm:FiniteMeasureSpace} Assume that $b$ is a locally integrable function that has vanishing mean oscillations along distant cubes,
$$\lim_{M\to \infty} \sup_{\substack{Q\in\mathcal{D}:\\\dist(0,Q) \geq M}} \osc=0.$$
Let $\mathcal{Q}$ be a connected collection of dyadic cubes.
Then, the operator $P_{b,\mathcal{Q}}:L^p(\mathbb{R}^d,\mu)\to L^p(\mathbb{R}^d,\mu)$ is compact if the operator $P_{b,\mathcal{R}}:L^p(\mathbb{R}^d,\mu)\to L^p(\mathbb{R}^d,\mu)$ is compact for every  subcollection $\mathcal{R}\subseteq \mathcal{Q}$ that can be partitioned as
$
\mathcal{R} = \mathcal{R}^{\operatorname{in}}\cup\mathcal{R}^{\operatorname{out}}
$
with
\begin{itemize}
\item (Coarse scales) $\mathcal{R}^{\operatorname{out}}\subseteq\mathcal{Q}\cap  \mathcal{D}_{\leq L_0},$ for some $L_0\in \bz,$ and
\item (Chain) $\mathcal{R}^{\operatorname{out}}:=\{Q_1\subseteq Q_2\subseteq \ldots\};$
\item (Fine scales)  $\mathcal{R}^{\operatorname{in}}\subseteq \mathcal{Q}\cap \mathcal{D}_{\geq L_0},$ for some $L_0\in \bz,$ and
\item (Finitely supported) $\mu(\cup_{Q\in\mathcal{R}^{\operatorname{in}}} Q)<+\infty.$
\item (Connectedness) The collection $\mathcal{R}^{\operatorname{in}}$ and the collection $\mathcal{R}^{\operatorname{out}}$ are both connected.
\end{itemize}
\end{lemma}
We use measure-theoretic arguments instead of metric arguments, so that the argument works in the measure-theoretic setting as well.
\begin{proof}[Proof of Lemma \ref{lem:OpNorm:FiniteMeasureSpace}] For each $M\in\bn$ let $\ca_M\subseteq \mathcal{D}$ be an auxiliary finite collection of cubes. Set $\Omega_M:=\bigcup_{A\in\ca_M} A$. Choose the auxiliary collections $\ca_M$ so that $\Omega_M\uparrow \br^d$ and, for compatibility with the metric definition for vanishing mean oscillation along distant cubes, so that $B(0,M)\subseteq \Omega_M$. For each $A\in\ca_M$ there is $k(A,M)\in \bz$ such that $A\in \mathcal{D}_{k(A,M)}.$ Define the finest and coarsest auxiliary scales
$$
L^*(M):=\max_{A\in\ca_M}k(A,M),\qquad L_*(M):=\min_{A\in\ca_M}k(A,M).
$$
Now, we partition 
$$
\mathcal{Q} =  \mathcal{Q}^{\mathrm{dist}}(M) \sqcup \mathcal{Q}^{\mathrm{in}}(M) \sqcup \mathcal{Q}^{\mathrm{out}}(M), 
$$ 
where the three collections are
\begin{align*}
    \mathcal{Q}^{\mathrm{dist}}(M) &:=\Big\{Q\in \mathcal{Q} : Q\subseteq \Omega_M^c\Big\}, \\
\mathcal{Q}^{\mathrm{in}}(M) &:=\Big\{Q\in \mathcal{Q} : \text{$Q\cap \Omega_M\neq \emptyset$ and $Q\subseteq A$ for some $A\in \ca_M$ with $Q\cap A\neq \emptyset$}\Big\}, \\
\mathcal{Q}^{\mathrm{out}}(M) &:=\Big\{Q\in \mathcal{Q} : \text{$Q\cap \Omega_M\neq \emptyset$ and $Q\supsetneq A$ for all $A\in \ca_M$ with $Q\cap A\neq \emptyset$}\Big\}.
\end{align*}
Notice that 
$$
\mathcal{Q}^{\operatorname{in}}(M)\subset \mathcal{D}_{\geq L^*(M)},\qquad \mathcal{Q}^{\operatorname{out}}(M)\subset \mathcal{D}_{\leq L_*(M)}.
$$
First, we consider the outer cubes. Let $\mathcal{Q}^{\mathrm{out}}_*(M)$ be the collection of all the minimal (with respect to the set containment) cubes $Q_*$ in $\mathcal{Q}^{\mathrm{out}}(M)$. By nestedness, the cubes $Q_*$ are disjoint. Moreover, for each $Q_*$ there is $A\in\ca_M$ so that $A\subsetneq Q_*$. The collection $\ca_M$ is finite. Thus, the collection $\mathcal{Q}^{\mathrm{out}}_*(M)$ is finite. Now, we decompose
$$
\mathcal{Q}^{\mathrm{out}}(M)=\bigsqcup_{Q_* \in \mathcal{Q}^{\mathrm{out}}_*(M)} \{Q\in \mathcal{Q}^{\mathrm{out}}(M) : Q\supseteq Q_*\}=:\bigsqcup_{Q_* \in \mathcal{Q}^{\mathrm{out}}_*(M)} \mathcal{Q}^{\mathrm{out}}_{Q_*}(M).
$$
Now, the collection $\mathcal{Q}^{\mathrm{out}}_{Q_*}(M)$ has all the claimed properties:
\begin{itemize}
\item (Coarse scales) We have $\mathcal{Q}^{\mathrm{out}}_{Q_*}(M)\subseteq \mathcal{Q}^{\mathrm{out}}(M)\subseteq  \mathcal{D}_{\leq L^*(M)}$.
\item (Chain order)  The collection $\mathcal{Q}^{\mathrm{out}}_{Q_*}(M)$ can be exhausted by an increasing sequence $Q_0:=Q_*\subseteq Q_1\subseteq \ldots$.
\item (Connectedness) If $\mathcal{Q}$ is connected, then so is $\mathcal{Q}^{\mathrm{out}}_{Q_*}(M)$. 
\end{itemize}
Also the collection $\mathcal{Q}^{\mathrm{in}}(M)$ has all the desired properties:
\begin{itemize}
\item (Fine scales) We have $\mathcal{Q}^{\mathrm{in}}(M)\subseteq  \mathcal{D}_{\geq L_*(M)}$.
\item (Supported on a set of finite measure)  We have
$$
\mu\Big(\bigcup_{Q\in \mathcal{Q}^{\mathrm{in}}(M)}Q\Big)\leq \mu\Big( \bigcup_{A\in\mathcal{A}_M} A \Big)\leq \mu(\Omega_M) < +\infty.
$$
\item (Connectedness) If $\mathcal{Q}$ is connected, then so is $\mathcal{Q}^{\mathrm{in}}(M)$. 
\end{itemize}
Finally, we observe that
$$\mathcal{Q}^{\mathrm{dist}}(M)\subseteq \{Q\in \mathcal{D}: \dist(0,Q)\geq M\}.$$
The corresponding decomposition for the paraproduct is
$$
P_\mathcal{Q}:=\sum_{Q\in\mathcal{Q}} P_Q:=\sum_{\mathcal{Q}^{\mathrm{in}}(M)}P_Q+\sum_{Q_* \in\mathcal{Q}^{\mathrm{out}}_*(M)} \Big( \sum_{Q\in \mathcal{Q}^{\mathrm{out}}_{Q_*}(M)}P_Q\Big)+\sum_{Q\in \mathcal{Q}^{\mathrm{dist}}(M)}P_Q.
$$
The operator norm of the term $\sum_{Q\in \mathcal{Q}^{\mathrm{dist}}(M)}P_Q$ tends to zero as $M\to \infty$ by the assumption on vanishing mean oscillation along distant cubes. Therefore, the operator $P_{\mathcal{Q}}$ is compact if the operators $P_{\mathcal{Q}^{\mathrm{in}}(M)}$ and $P_{\mathcal{Q}^{\mathrm{out}}_{Q_*}(M)}$
are both compact.
\end{proof}

\begin{lemma}[Outer cubes of bounded measure]\label{lem:OpNorm:OuterCubes} Let $\mu$ be a locally finite Borel measure on $\br^d$, and let $\mathcal{Q}$ be a \emph{connected} collection of dyadic cubes. Assume that $b\in\bmo^p(\mathbb{R}^d,\mathcal{Q},\mu)$. Assume that $\mathcal{Q}$ is such that (1) $\mathcal{Q}\subseteq \mathcal{D}_{\leq L_0}$ for some $L_0\in \bz$, (2)  $\mathcal{Q}:=\{Q_1\subseteq Q_2\subseteq \ldots\}$ is an increasing sequence of cubes, (3) $\mu(Q)\sim 1$ for all $Q\in \mathcal{Q}$. 
Then, the operator $P_{b,\mathcal{Q}}$ is compact. 
\end{lemma}
\begin{proof}[Proof of Lemma \ref{lem:OpNorm:OuterCubes}]Let $L\leq L_0$ and by the condition $(1)$ decompose
$$
\sum_{Q\in\mathcal{Q}}P_Q=\sum_{Q\in \mathcal{Q}\cap \mathcal{D}_{[L,L_0]}}P_Q+\sum_{Q\in \mathcal{Q}\cap \mathcal{D}_{\leq L}}P_Q.
$$
By the condition $(2),$ the operator $\sum_{Q\in \mathcal{Q}\cap \mathcal{D}_{[L,L_0]}}P_Q$ is of finite rank (and hence compact) for every $L.$
To control the error, observe first that by $(3):$
$$
\int_{\bigcup_{k=1}^\infty Q_k} \abs{b}\dmu\lesssim \int_{Q_1} \abs{b}+\norm{b}_{\bmo^p(\mathbb{R}^d,\mathcal{Q},\mu)}<+\infty.
$$
This together with $\mu(Q_k)\lesssim 1,$ for $k=1,\dots,$ (by $(3)$) implies that
$$
\lim_{L\to -\infty} \sup_{m,n\leq L} \abs{\angles{b}_{Q_m}-\angles{b}_{Q_n}}=0.
$$
Thus, by the connectedness of $\mathcal{Q}$ and Lemma \ref{lem:OpNorm:connected}:
$$
\lim_{L\to-\infty} \norm{\sum_{Q\in \mathcal{Q}\cap \mathcal{D}_{\leq L}}P_Q}_{L^p(\mathbb{R}^d,\mu)\to L^p(\mathbb{R}^d,\mu)}\lesssim  \lim_{L\to -\infty} \sup_{m,n\leq L} \abs{\angles{b}_{Q_m}-\angles{b}_{Q_n}}=0.
$$
\end{proof}
\begin{lemma}[Inner cubes of bounded measure]\label{lem:OpNorm:InnerCubes}  Let $\mu$ be a locally finite Borel measure on $\br^d$, and let $\mathcal{Q}$ be a \emph{connected} collection of dyadic cubes. Assume that $b\in\bmo^p(\mathbb{R}^d,\mathcal{Q},\mu)$. Assume that $\mathcal{Q}$ is such that (1) $\mathcal{Q}\subseteq \mathcal{D}_{\geq L_0}$ for some $L_0\in \bz$, (2) $\mu(\bigcup_{Q\in\mathcal{Q}} Q)<+\infty$, (3) $\mu(Q)\sim 1$ for all $Q\in \mathcal{Q}$. 
Then, the operator $P_\mathcal{Q}$ is compact.
\end{lemma}
\begin{proof}[Proof of Lemma \ref{lem:OpNorm:InnerCubes}] Let $L\geq L_0$ and by the assumption $(1)$ decompose
$$
\sum_{Q\in\mathcal{Q}}P_Q=\sum_{Q\in \mathcal{Q}\cap \mathcal{D}_{[L_0,L]}}P_Q+\sum_{Q\in \mathcal{Q}\cap \mathcal{D}_{\geq L}}P_Q.
$$
Since each $\mathcal{D}_K$, with $K\in\bz$, is a collection of pairwise disjoint sets, it follows from the assumptions $(2)$ and $(3)$ that each collection $\mathcal{Q}\cap\mathcal{D}_{K}$ is finite and hence the collection
$
\mathcal{Q}\cap \mathcal{D}_{[L_0,L]}=\bigcup_{K=L_0}^L \mathcal{Q}\cap \mathcal{D}_{K}$ is finite for every $L\geq L_0$.
Therefore, the operator $\sum_{Q\in \mathcal{Q}\cap \mathcal{D}_{[L_0,L]}}P_Q$ is of finite rank for every $L\geq L_0$. To control the error term $\sum_{Q\in \mathcal{Q}\cap \mathcal{D}_{\geq L}}P_Q$, we estimate as follows:
\begin{align*}
    &\lim_{L\to\infty} \norm{\sum_{Q\in \mathcal{Q}\cap \mathcal{D}_{\geq L}}P_Q}_{L^p(\mathbb{R}^d,\mu)\to L^p(\mathbb{R}^d,\mu)}
\lesssim \lim_{L\to\infty} \norm{ \sup_{\substack{ R,S\in\mathcal{Q}\cap\mathcal{D}_{\geq L} \\ R\subset S }}\abs{\angles{b}_R-\angles{b}_S}1_S}_{L^p(\mathbb{R}^d,\mu)}\\
&= \norm{\lim_{L\to\infty} \sup_{\substack{ R,S\in\mathcal{Q}\cap\mathcal{D}_{\geq L} \\ R\subset S }}\abs{\angles{b}_R-\angles{b}_S}1_S}_{L^p(\mathbb{R}^d,\mu)} =\norm{0}_{L^p(\mathbb{R}^d,\mu)}.
\end{align*}
Above we used Lemmas \ref{lem:OpNorm:Carleson} and \ref{lem:OpNorm:connected} together with $\mu(Q)\sim 1$ for all $Q\in\mathcal{Q}$ in the first bound;
then, the dominated convergence theorem to interchange the order of limit and integration; and finally, the Lebesgue differentiation theorem to calculate the pointwise limit.
\end{proof}

\section{On VMO spaces}\label{section:onvmo}
\subsection{Masswise versus lengthwise VMO spaces}\label{section:relation}
For brevity we write $\omega_p(b,Q)$ for the mean oscillation
$$
\omega_p(b,Q):= \left(\inf_c\frac{1}{\mu(Q)}\int_Q |b-c|^p \dmu\right)^{1/p}.
$$
On a quasi-metric measure space equipped with a dyadic system, the $\vmo$ space can be defined either using lengths or measures:

\begin{definition}[Length-version of a $\vmo$ space]\label{definition:vmo_length}A locally integrable function $b$ belongs to $\vmo^p_{\text{length}}$ if it satisfies the following conditions:
\begin{enumerate}
\item (Small cubes) We have
$$
\limsup_{\substack{\ell(Q)\to 0:\\ \dist(x_0,Q)\lesssim 1}} \omega_p(b,Q)=0.
$$
\item (Large cubes) We have
$$
\limsup_{\substack{\ell(Q)\to +\infty:\\ \dist(x_0,Q)\lesssim 1}} \omega_p(b,Q)=0.
$$
\item (Heavy cubes) We have
$$
\limsup_{\dist(x_0,Q)\to +\infty} \omega_p(b,Q)=0.
$$
\end{enumerate}
\end{definition}

\begin{definition}[Measure-version of a $\vmo$ space]\label{definition:vmo_mass}A locally integrable function $b$ belongs to $\vmo^p_{\text{mass}}$ if it satisfies the following conditions:
\begin{enumerate}
\item (Light cubes) We have
$$
\limsup_{\substack{\mu(Q)\to 0:\\ \dist(x_0,Q)\lesssim 1}} \omega_p(b,Q)=0.
$$
\item (Heavy cubes) We have
$$
\limsup_{\substack{\mu(Q)\to +\infty:\\ \dist(x_0,Q)\lesssim 1}} \omega_p(b,Q)=0.
$$
\item (Distant cubes) We have
$$
\limsup_{\dist(x_0,Q)\to +\infty} \omega_p(b,Q)=0.
$$
\end{enumerate}
\end{definition}
\begin{remark}
Each definition is independent of the choice of the reference point $x_0$, which can be chosen in particular to be the origin on the Euclidean space $\br^d$. By using the condition 
$
\limsup_{\dist(x_0,Q)\to +\infty} \omega_p(b,Q)=0
$, the constraint $ \dist(x_0,Q)\lesssim 1$ can be included in the other conditions without loss of generality.
\end{remark}

 Recall that a standard dyadic system $\cd$ on $\br^d$ is defined by $\cd:=\bigcup_{k\in \bz} \cd_k$, where $(\cd_k)_{k\in\bz}$ is any sequence of refining partitions $\cd_k$ consisting of cubes of side-length $2^{-k}$. For simplicity, we focus on the Euclidean space $(\br^d,\abs{\cdotroomy})$ equipped with a generic Borel measure $\mu$ and a standard dyadic system $\cd$. However,
 we write the proofs in a way that they can be adapted to more general quasi-metric measure spaces.
 The following result characterizes when the lengthwise and measurewise VMO spaces coincide.

\begin{theorem}[Relation between the measure and length versions of VMO]\label{theorem:relation}Let $\mu$ be a locally finite Borel measure and $\cd$ a standard dyadic system on $\br^d$. 

\begin{enumerate}
\item We have 
\begin{align*}
    \vmo^p_{\text{length}}(\br^d,\cd,\mu)\subseteq \vmo^p_{\text{mass}}(\br^d,\cd,\mu).
\end{align*}
\item We have
\begin{align}\label{eq:thmVMO2}
    \vmo^p_{\text{mass}}(\br^d,\cd,\mu)\subseteq \vmo^p_{\text{length}}(\br^d,\cd,\mu)
\end{align}
if and only if $\mu(D) = +\infty$ for every dyadic orthant $D$ on which $\mu\restriction_D$ is not a single point mass.
\end{enumerate}
\end{theorem}
\begin{proof} 

{\it ``(1).''} The inclusion
$
 \vmo^p_{\text{length}}\subseteq \vmo^p_{\text{mass}}
$
follows by Items $(1)$ and $(2)$ of Corollary \ref{corollary:implications}.  

{\it ``(2), $\impliedby$.''} We may assume without loss of generality that $\mu\restriction_D$ is not a single point mass on any dyadic orthant $D$. Indeed, if $\mu\restriction_D$ a single point mass on $D$, then $\omega_p(b,Q)\equiv 0$ for all $b\in L^1_{\text{loc}}$ and $Q\subseteq D$. Thus,
$$
\limsup_{\cdots : Q\in\cd} \omega_p(b,Q) = \limsup_{\cdots : Q\in\cd'} \omega_p(b,Q),
$$ 
where $\cd'\subseteq \cd$ is defined by
$$
\cd':=\{Q\in \cd : \text{$Q\subseteq D$ for some $D$ such that $\mu\restriction_D$ is not a single point mass}\}.$$
Now, the inclusion
$
  \vmo^p_{\text{mass}}\subseteq \vmo^p_{\text{length}}
$
follows by Item $(3)$ of Corollary \ref{corollary:implications} and by Lemma \ref{lemma:smallheavy}.

{\it ``(2), $\implies$.''} From the inclusion
$
  \vmo^p_{\text{mass}}\subseteq \vmo^p_{\text{length}}
$
it follows that for every $b\in L^1_{\text{loc}}$ such that
$$
\limsup_{\substack{\mu(Q)\to 0:\\ \dist(x_0,Q)\lesssim 1}} \omega_p(b,Q)=0 \quad\text{and}\quad \limsup_{\substack{\dist(x_0,Q)\to +\infty}} \omega_p(b,Q)=0
$$
we have the implication
$$
    \limsup_{\substack{\mu(Q)\to \infty:\\ \dist(x_0,Q)\lesssim 1}} \omega_p(b,Q)=0 \Longrightarrow \limsup_{\substack{\ell(Q)\to \infty:\\ \dist(x_0,Q)\lesssim 1}} \omega_p(b,Q)=0.
$$
and hence by Lemma \ref{lem:SharpnessLargeCubes} we have $\mu(D) = +\infty$ for every dyadic orthant $D$ on which $\mu\restriction_D$ is not a single point mass.
\end{proof}

\subsubsection*{Measure versus length}With the aim of relating vanishing oscillations along light versus small cubes, and along heavy versus large cubes, we begin by studying how length and measure of dyadic cubes are related. 

The proofs are written in a way that they can be adapted to more general quasi-metric measure spaces. In doing so, we have intentionally avoided simplifications available in the Euclidean space $\br^d$ (for example, some arguments could be simplified using the fact that every closed ball in $\br^d$ is compact). Instead, the proofs rely largely on the properties that (1) each ball can contain at most finitely many disjoint balls of a fixed radius, and that (2) for each dyadic cube $Q$ there are balls $B'_Q\subseteq Q\subseteq B_Q$ such that $\ell(Q)\sim r_{B'_Q} \sim r_{B_Q}$.
\begin{lemma}[Length as a function of measure]Let $\cd$ be a dyadic system. Let $\mu$ be a Borel measure such that $0<\mu(Q)<+\infty$ for all $Q\in\cd$. Then the following hold:
\begin{enumerate}[label=(\roman*)]
\item We have
$$
\lim_{\substack{\mu(Q)\to 0:\\\dist(0,Q)\lesssim 1}}\ell(Q)=0.
$$
\item We have
$$
\lim_{\substack{\mu(Q)\to +\infty:\\\dist(0,Q)\lesssim 1}}\ell(Q)=+\infty.
$$
\end{enumerate}
\end{lemma}
\begin{proof}
{\it`` (i) ''.} Assuming $\dist(0,Q)\leq R$ for a fixed $R>0$, for each $\epsilon>0$ we need to find $\delta:=\delta(\epsilon)$ such that $\mu(Q)<\delta$ implies $\ell(Q)<\epsilon$ or, by contraposition, such that $\ell(Q)\geq \epsilon$ implies $\mu(Q)\geq \delta$. Let $$
\cq:=\{Q\in \cd: \dist(0,Q)\leq R, \ell(Q)\geq \epsilon\}.
$$ Let $Q_i$ be the minimal cubes in $\cq$. Note that  $Q_i$ are pairwise disjoint by dyadic nestedness. For each $Q_i$ choose a ball $B_i$ such that $B_i\subseteq Q_i$ and $r_{B_i}\sim \epsilon$ and $\dist(B(0,R),B_i)\lesssim \epsilon$. Hence $\dist(0,B_i)\lesssim R+\epsilon$. Note that $B_i$ are finitely many because any ball can contain only finitely many disjoint balls of a fixed radius. Thus, for every $Q\in\cq$ we have
$$
\mu(Q)\geq \min_{i} \mu(Q_i)=:\delta(\epsilon)>0.
$$

{\it `` (ii) ''.}  Assuming $\dist(0,Q)\leq R$ for a fixed $R>0$, for each $L>0$ we need to find $M(L)$ such that $\mu(Q)>M$ implies $\ell(Q)>L$ or, by contraposition, such that $\ell(Q)\leq L$ implies $\mu(Q)\leq  M$. Let $$
\cq:=\{Q\in \cd: \dist(0,Q)\leq R, \ell(Q)\leq L\}.$$
 
For each $Q\in\cq$ choose a dyadic ancestor $A_Q$ such that $A_Q\supseteq Q$ and $\ell(A_Q)\sim L$. Let $A_i$ be the maximal cubes of $\{A_Q\}_{Q\in\cq}$. Note that $A_i$ are disjoint by dyadic nestedness. Note as before that $A_i$ are finitely many because any ball can contain only finitely many disjoint balls of a fixed radius. 
Thus, for every $Q\in\cq$ we have
$$
\mu(Q)\leq \mu(A_Q)\leq \min_{i} \mu(A_i)=:M(L)<+\infty.
$$
\end{proof}

Let $Q^{(1)}$ denote the dyadic parent of a dyadic cube $Q\in \cd$. The $k$-th ancestor $Q^{(k)}$ of a dyadic cube $Q$ is defined recursively by $Q^{(k)}:= (Q^{(k-1)})^{(1)}$. The {\it dyadic orthant} $Q^{(+\infty)}$ containing a dyadic cube $Q\in\cd$ is defined by $Q^{(+\infty)}:=\bigcup_{k=1}^\infty Q^{(k)}$. By construction, dyadic orthants are disjoint and each dyadic cube is contained in a unique dyadic orthant. Hence, the dyadic system  can be decomposed as $\cd=\bigsqcup_{i} \cd_i$ into disjoint collections $\cd_i:=\{Q\in \cq : Q\subseteq D_i\}$  corresponding to its dyadic orthants $\{D_i\}:=\{Q^{(+\infty)}\}_{Q\in\cd}$. A dyadic system of cubes in the $d$-dimensional Euclidean space can have at most $2^d$ dyadic orthants.

\begin{lemma}[Measure as a function of length]\label{lemma:measurevialength}
Let $\mu$ be a locally finite Borel measure and $\cd$ be a dyadic system. Then the following hold:
\begin{enumerate}[label=(\roman*)]
\item We have
$$
\lim_{\substack{\ell(Q)\to 0:\\\dist(0,Q)\lesssim 1}}\mu(Q)=0
$$
if and only if $\mu$ has no point mass.
\item We have
$$
\lim_{\substack{\ell(Q)\to +\infty:\\\dist(0,Q)\lesssim 1}}\mu(Q)=+\infty
$$
if and only if $\mu(D)=+\infty$ for every dyadic orthant $D$.
\end{enumerate}
\end{lemma}
\begin{proof}
{\it ``(i), $\implies$.''} Let $x\in \br^d$. Choose a decreasing sequence $Q_1\supseteq Q_2\supseteq \cdots$ such that $\lim_{n\to \infty} \ell(Q_n)=0$. Now $x=\bigcap_{n=1}^\infty Q_n$ and by the monotone convergence theorem $\mu(\{x\})=\mu\big(\bigcap_{n=1}^\infty Q_n\big)=\lim_{n\to \infty} \mu(Q_n)=0.$

{\it ``(i), $\impliedby$.''} Proof by contraposition. Assume that for some $R>0$ and for some $\epsilon>0$ for each $\delta>0$ there is a cube $Q$ with $\dist(0,Q)\leq R$ such that $\ell(Q)\leq \delta$ but $\mu(Q)\geq \epsilon$. Consider the collection 
$$
\cq_0:=\{Q\in \cd: \dist(0,Q)\leq R, \mu(Q)\geq \epsilon, \ell(Q)\leq 1 \}.
$$
We claim that there is a sequence  $Q_1\supseteq Q_2\supseteq \cdots $ with $\ell(Q_n)\to 0$ and $Q_n\in\cq_0$. Note that the set $\bigcap_{n=1}^\infty Q_n$ is either empty or a singleton. By the monotone convergence theorem, $\mu\big(\bigcap_{n=1}^\infty Q_n\big)= \lim_{n\to \infty} \mu(Q_n)\geq  \epsilon$ and hence $\bigcap_{n=1}^\infty Q_n=\{x\}$ for some $x\in \br^d$. We prove the claim next.

For each $k=1,2,\ldots$ we define recursively $\cq_{k+1}$ as the collection of all the maximal strict subcubes of $Q\in\cq_k$. Note that each $\cq_k$ is non-empty by the assumption of the contrapositive (which guarantees that there are arbitrarily small cubes). Moreover, note that each $\bigcup_{Q\in\cq_k} Q$ is contained in a ball $B(0,r_B)$ of radius $r_B\sim R+1$, each $\cq_k$ is disjoint, and $\mu(Q)\geq \epsilon$ for all $Q\in\cq_k$, which together imply that 
$$
\sup_{k} \# \cq_k \leq \frac{\mu(B(0,r_B))}{\epsilon}<+\infty.
$$ 
Define $E_k:=\bigcup_{Q\in\cq_k} Q$. Note that $E_k\supsetneq E_{k+1}$ by construction. 
We note that 
$$
\norm{\sum_{k=1}^\infty 1_{E_k}}_\infty \mu(B(0,r_B))\geq \int  \sum_{k=1}^\infty 1_{E_k} \geq \inf \#\cq_n \sum_{n=1}^\infty \epsilon=+\infty.
$$ 
Therefore, the function $\sum_{k=1}^\infty 1_{E_k}$ is unbounded and hence for every $N$ there is $x_N$ such that $x_N\in E_N$, which by construction means that for every $N$ there is a sequence $Q_1\supseteq Q_2\supseteq \cdots \supseteq Q_N$ with $Q_n\in \cq_n$ for every $n=1,\ldots,N$. We note that if a rooted tree is such that its every node has a finite degree and the tree has arbitrary long paths, then the tree must have an infinite long path. This follows from observing that if a node has arbitrary long paths, then one of its finitely many descendant nodes must also have arbitrary long paths (thanks to Jaakko Sinko for pointing this out).  Therefore, there is a sequence  $Q_1\supseteq Q_2\supseteq \cdots $ with $Q_n\in \cq_n$ for all $n=1,2,\ldots$. 

{\it ``(ii), $\implies$.''} Let $D$ be a dyadic orthant. By definition, $D=\bigcup_{k=1}^\infty Q^{(k)}$ for some $Q\in\cd$. By the monotone convergence theorem together with the assumption, $\mu(D)=\lim_{k\to \infty} \mu(Q^{(k)})=+\infty.$

{\it ``(ii), $\impliedby$.''} Fix $R>0$. For each $M>0$ we need to find $L:=L(M)>$ such that for every $Q$ with $\dist(0,Q)\leq R$ having  $\ell(Q)\geq L$ implies $\mu(Q)\geq M$. First, we expand the ball $B(0,R)$ by considering $B(0,R+1)$ and cover $B(0,R+1)$ by finitely many $Q_i\in\cd$. By assumption for each $i$ we can find $k_i$ such that $\mu(Q_i^{(k_i)})\geq M$. Set $L(M):=\max_i  \ell(Q_i^{(k_i)}).$ Let $Q\in\cd$ with $\ell(Q)\geq L$ and $\dist(0,Q)\leq R$. Since $Q$ intersects $B(0,R+1)=\bigcup_i Q_i$, it intersects a cube $Q_i$ and hence with greater reason also $Q_i^{(k_i)}$. By dyadic nestedness, alternatively either $Q\subsetneq Q_i^{(k_i)}$ or $Q\supseteq Q_i^{(k_i)}$. Since $\ell(Q)\geq L\geq \ell(Q_i^{(k_i)})$, only the latter alternative is possible and hence $\mu(Q)\geq \mu(Q_i^{(k_i)})\geq M$.
\end{proof}

The limit relations between measure and length are used through the following elementary fact about upper limits:
\begin{lemma}[Substitution principle for upper limits]Let $\cd$ be a set and $a,b,\omega:\cd\to[0,+\infty]$ be functions. Then the following hold:
\begin{enumerate}[label=(\roman*)]
\item If $
\lim_{a(Q)\to +\infty} b(Q)=+\infty,
$
then
$
\limsup_{a(Q)\to \infty} \omega(Q)\leq \limsup_{b(Q)\to \infty} \omega(Q).
$
\item 
If $
\lim_{a(Q)\to 0} b(Q)=0,
$
then
$
\limsup_{a(Q)\to 0} \omega(Q)\leq \limsup_{b(Q)\to 0} \omega(Q).
$
\end{enumerate}
\end{lemma}
\begin{proof}{\it ``(i).''} The assumption $
\lim_{a(Q)\to +\infty} b(Q)=+\infty$ means that for each $b>0$ there is $a(b)>0$ such that $a_Q\geq a(b)$ implies $b(Q)>b$. Therefore, for every $b>0$ for every $a\geq a(b)$ we have
$$
\sup_{\substack{Q\in \cd: \\ a(Q)\geq a}} \omega(Q)\leq \sup \sup_{\substack{Q\in \cd: \\ b(Q)\geq b}} \omega(Q).
$$
Taking first the limit as $a\to +\infty$ and then the limit as $b\to +\infty$ yields the claim.
\end{proof}
This elementary fact together with the limit relations between measure and length of cubes yields the following corollaries:
\begin{corollary}\label{corollary:implications}Let $\mu$ be a locally finite Borel measure and $\cd$ a dyadic system on $\br^d$. Then the following hold:
\begin{enumerate}
\item We have that
$$\limsup_{\substack{\ell(Q)\to 0:\\ \dist(0,Q)\lesssim 1}} \omega_p(b,Q)=0\quad \implies \quad \limsup_{\substack{\mu(Q)\to 0:\\ \dist(0,Q)\lesssim 1}} \omega_p(b,Q)=0.$$
\item We have that
$$\limsup_{\substack{\ell(Q)\to +\infty:\\ \dist(0,Q)\lesssim 1}} \omega_p(b,Q)=0\quad\implies \quad \limsup_{\substack{\mu(Q)\to +\infty:\\ \dist(0,Q)\lesssim 1}} \omega_p(b,Q)=0.$$
\item Assuming $\mu(D)=+\infty$ for all dyadic orthants, we have
$$\limsup_{\substack{\mu(Q)\to +\infty:\\ \dist(0,Q)\lesssim 1}} \omega_p(b,Q)=0\quad\implies \quad \limsup_{\substack{\ell(Q)\to +\infty:\\ \dist(0,Q)\lesssim 1}} \omega_p(b,Q)=0.$$
\item Assuming $\mu$ has no point masses, we have
$$\limsup_{\substack{\mu(Q)\to 0:\\ \dist(0,Q)\lesssim 1}} \omega_p(b,Q)=0\quad\implies \quad \limsup_{\substack{\ell(Q)\to 0:\\ \dist(0,Q)\lesssim 1}} \omega_p(b,Q)=0.$$
\end{enumerate}
More generally, the same result holds for any non-negative coefficients $\omega(Q)$ in place of the mean oscillation $\omega_p(b,Q)$.
\end{corollary}
Remarkably, this result relies solely on the nature of measure and distance of dyadic cubes, not on the nature of mean oscillations. The result is sharp for oscillations along large cubes:

\begin{lemma}[Sharpness for oscillations along large cubes]\label{lem:SharpnessLargeCubes} Assume that $\mu$ is oscillatorily non-degenerate on a dyadic orthant $D$ in the sense that $\mu\restriction_D$ is not a single point mass.  If for every $b\in L^1_{\text{loc}}$ we have
\begin{equation}
\label{assumption:temp1}
\limsup_{\substack{\mu(Q)\to +\infty:\\ Q \subseteq D,\\ \dist(0,Q)\lesssim 1}} \omega_p(b,Q)=0\quad\implies \quad \limsup_{\substack{\ell(Q)\to +\infty:\\Q\subseteq D,\\ \dist(0,Q)\lesssim 1}} \omega_p(b,Q)=0,\end{equation}
then $\mu(D)=+\infty$ for the dyadic orthant $D$.
\end{lemma}
\begin{remark} If the underlying measure $\mu$ is a single point mass on a dyadic orthant $D$, then we have $\omega_p(b,Q)\equiv 0$ for all $b\in L^1_{\text{loc}}$ and all $Q\subseteq D$ and so in this case the mean oscillation carries no information.
\end{remark}
\begin{proof} Assume that $D$ is a dyadic orthant such that $\mu\restriction_D$ is not a single point mass.
Proof by contraposition. Assume $\mu(D)<+\infty$.  Since $\mu\restriction_D$ is not a single point mass, there are dyadic cubes $Q_1\subsetneq Q_2\subseteq D$ such that $\mu(Q_1)>0$ and $\mu(Q_2\setminus Q_1)>0$. Define $b$ by setting 
$$b:=\begin{cases}1 & \text{on $Q_1$,}\\
-1 &\text{on $Q_2\setminus Q_1$,}\\
0&\text{elsewhere}.\end{cases}$$
We observe the following:
\begin{itemize} 
\item (Heavy cubes) We have $$
\limsup_{\mu(Q)\to +\infty} \omega_p(b,Q)=0
$$ 
because if $Q\cap \supp(b)\neq\emptyset$, then $Q\subseteq D$ and there is no cube $Q\subseteq D$ such that $\mu(Q)>\mu(D)$. 
\item (Large cubes) We have
\begin{equation*}
    \begin{split}
&\limsup_{\ell(Q)\to +\infty}\omega_p(b,Q)\gtrsim\inf_{Q\subseteq D}\inf_c \big(\frac{1}{\mu(Q)} \int_Q |b-c|\dmu\big)\\
&=\inf_{Q\subseteq D}\inf_c \big(\frac{1}{\mu(Q)} \big(\int_{Q_1} |1-c|\dmu+\int_{Q_2\setminus Q_1} |1+c|\dmu\big)\big)\\
&\geq \frac{1}{\mu(D)}\min\{\mu(Q_1),\mu(Q_2\setminus Q_1)\}\inf_c (|1-c|+|1+c|) >0.
 \end{split}
\end{equation*}
\end{itemize}
\end{proof}

\begin{remark}The function $b:=1_{Q_1}-1_{Q_2\setminus Q_1}$ defined above  has also the following vanishing mean oscillations:
\begin{itemize}
\item (Light cubes) We have $$\limsup_{\mu(Q)\to 0} \omega_p(b,Q)=0$$ because if $\mu(Q)\leq \mu(Q_1)$ and $Q\cap \supp(b)\neq \emptyset$, then either $Q\subseteq Q_1$ or $Q\subseteq Q_2\setminus Q_1$ and hence $\omega_p(b,Q)=0$ because $b$ is constant on $Q_1$ and on $Q_2\setminus Q_1$.

\item (Distant cubes) We have $$\limsup_{\dist(Q,0)\to +\infty} \omega_p(b,Q)=0$$ because if $\dist(Q,0)\gtrsim \ell(Q_2)+\dist(Q_2,0)$, then $Q\cap \supp(b)=\emptyset$.
\end{itemize} 
\end{remark}
Whereas  Corollary \ref{corollary:implications} is sharp for large cubes, it fails to be sharp for oscillations along small cubes. It turns out that the assumption that $\mu$ has no point masses is not needed for oscillations because mass concentrates at points and oscillations at a mass point vanish, which leads to the sharp result captured by the following lemmas. 

\begin{lemma}[Mass concentrates at points]\label{lemma:concentration} Let $\cd$ be a dyadic system, and let $\mu$ be a finite Borel measure. Then for each $\epsilon>0$ there exists $\delta>0$ such that if $\ell(Q)\leq \delta$, then $Q$ satisfies $\mu(Q)\leq \epsilon$ or $Q$ has exactly one point mass $x_Q$ and satisfies $\mu(Q\setminus \{x_Q\})\leq \epsilon$.
\end{lemma}
\begin{proof}Since $\mu$ is finite, it has countably many point masses $\{x_n\}$. We decompose
$$
\mu=\mu-\sum_{n=1}^\infty \delta_{x_n}+\sum_{n=1}^\infty \delta_{x_n}=:\mu_c+\mu_d.
$$
Let $\epsilon>0$. Since $\mu_c:=\mu-\sum_{n=1}^\infty \delta_{x_n}$ has no point mass, by Lemma \ref{lemma:measurevialength} we have
$$
\lim_{\ell(Q)\to 0} \mu_c(Q)=0
$$
and hence there is $\delta_1>0$ such that $\mu_c(Q)\leq \epsilon$ if $\ell(Q)\leq \delta_1$.

Since $\sum_{n=1}^\infty \mu(x_n)\leq \mu(\br^d)<+\infty$, there is $N$ such that $\sum_{n=N}^\infty \mu(x_n)\leq \epsilon$. Set $$\delta_2:=\frac{1}{2}\min_{n\neq m, n,m\leq N} \abs{x_n-x_m},$$ so that if $\ell(Q)\leq \delta_2$, then $Q$ can contain at most one $x_n$, $n=1,\ldots,N$. Set $\delta:=\min\{\delta_1,\delta_2\}.$

Now, each cube $Q$ with $\ell(Q)\leq \delta$ satisfies one of the following cases:

\begin{itemize}
\item $Q$ contains exactly one point $x_Q\in \{x_n\}_{n=1}^N$. Then $\mu(Q\setminus \{x_Q\})\leq \mu_c(Q)+ \sum_{n=N}^\infty \mu(x_n)\leq 2 \epsilon.$
\item $Q$ contains no point $x_Q\in \{x_n\}_{n=1}^N$. Then $\mu(Q)\leq \mu_c(Q)+ \sum_{n=N}^\infty \mu(x_n)\leq 2 \epsilon.$
\end{itemize}
\end{proof}

\begin{lemma}[Oscillations vanish along small heavy cubes]\label{lemma:smallheavy}Let $\mu$ be a locally finite Borel measure and $b\in L^1_{\text{loc}}$. Assume that $\int_Q \abs{b}^p \dmu <+\infty$ for all $Q\in\cd$. Then we have

$$\limsup_{\substack{\mu(Q)\to 0:\\ \dist(0,Q)\lesssim 1}} \omega_p(b,Q)=0\quad\implies \quad \limsup_{\substack{\ell(Q)\to 0:\\ \dist(0,Q)\lesssim 1}} \omega_p(b,Q)=0.$$
\end{lemma}
\begin{remark}
If $b\in \bmo^p$, then $\int_Q \abs{b}^p \dmu <+\infty$ for all $Q\in\cd$.
\end{remark}
\begin{proof}We need to prove that
$$
\limsup_{\substack{\ell(Q)\to 0:\\ \dist(0,Q)\lesssim 1}} \omega_p(b,Q)=0.
$$
Assuming $ \limsup_{\substack{\mu(Q)\to 0:\\ \dist(0,Q)\lesssim 1}} \omega_p(b,Q)=0,$ we have
$$
\limsup_{\substack{\ell(Q)\to 0:\\ \dist(0,Q)\lesssim 1}} \omega_p(b,Q)=0
\quad\iff\quad
\limsup_{\substack{\ell(Q)\to 0:\\ \dist(0,Q)\lesssim 1\\ \mu(Q)\gtrsim 1}} \omega_p(b,Q)=0
$$

We note that the cubes under consideration have side length $\ell(Q)\lesssim 1$ and distance $\dist(0,Q)\lesssim 1$ and are hence contained in some ball $B$. Thus, $\mu(Q)=\mu\restriction_B(Q)$ and $\mu\restriction_B(\br^d)=\mu(B)<+\infty$. By replacing $\mu$ by the restriction $\mu\restriction_B$, we may assume without loss of generality that $\mu(\br^d)<+\infty$. From now on the condition $\dist(0,Q)\lesssim 1$ is used only via the consequent condition $\mu(\br^d)<+\infty$. In summary, writing out the implicit constant, we need to prove that $$
\lim_{\gamma\to 0}\sup_{Q\in \cd: \mu(Q)\geq \delta_1,\ell(Q)\leq \gamma } \left(\inf_c\frac{1}{\mu(Q)}\int_Q |b-c|^p \dmu\right)^{1/p}=0
$$
for each $\delta_1>0$.
In what follows, we write explicitly  how parameters depend on other parameters as the choices are somewhat delicate and need care to avoid a circular choice.

Write
$$
\Omega:=\bigcup_{Q\in \cd: \mu(Q)\geq \delta_1, \ell(Q)\leq 1}Q.
$$
The set $\Omega$ is the union of the finitely many disjoint maximal cubes $Q^*\in \{Q\in \cd: \mu(Q)\geq \delta_1, \ell(Q)\leq 1\}$. (The maximal cubes exist because of the condition $\ell(Q)\leq 1$, they are disjoint because of maximality, and they are finitely many because of the condition $\mu(Q)\geq \delta_1$, disjointness, and the condition $\mu(\br^d)<+\infty$.)  Now, by assumption,
$$
\int_{\Omega} \abs{b}^p \dmu= \sum_{Q^*} \int_{Q^*} \abs{b}^p  <+\infty.
$$
By absolute continuity, there exists $\delta_2$  such that
$$
\int_E \abs{b}^p \dmu < a(\epsilon,\delta_1)
$$ 
whenever $\mu(E)\leq \delta_2$ and $E\subseteq \Omega$, where $a(\epsilon,\delta_1)$ is to be chosen later.  (Mnemonic: ``a'' stands for absolute continuity.)

By concentration of mass at points (Lemma \ref{lemma:concentration}), there exists $\gamma>0$ such that if $\ell(Q)\leq \gamma$, then
either $\mu(Q)\leq c(\epsilon,\delta_1,\delta_2)$ or there is exactly one point mass $x_Q\in Q$ and $\mu(Q\setminus \{x_Q\})\leq c(\epsilon,\delta_1,\delta_2),$ where $c(\epsilon,\delta_1,\delta_2)$ is to be chosen later. (Mnemonic: ``c'' stands for concentration of measure.) 
We now impose that $c(\epsilon,\delta_1,\delta_2)<\delta_1$, so that in fact the case $\mu(Q)\leq c(\epsilon,\delta_1,\delta_2)$ is absent from the limit superior $\lim_{\gamma\to 0}\sup_{Q\in \cd: \mu(Q)\geq \delta_1,\ell(Q)\leq \gamma }$. Moreover, we impose also that $c(\epsilon,\delta_1,\delta_2)<\delta_2$, so that the case $\mu(Q\setminus \{x_Q\})\leq c(\epsilon,\delta_1,\delta_2)$ can be controlled by absolute continuity.

Now, we estimate the oscillation as follows:
\begin{equation*}
\begin{split}
&\left(\inf_c\frac{1}{\mu(Q)}\int_Q |b-c|^p \dmu\right)^{1/p}\\
&\leq \left(\frac{1}{\mu(Q)}\int_{\{x_Q\}} |b-b(x_Q)|^p \dmu\right)^{1/p}+ \left(\frac{1}{\mu(Q)}\int_{Q\setminus \{x_Q\}}\abs{b}^p\dmu\right)^{1/p}\\
&+\left(\frac{\mu(Q\setminus \{x_Q\})}{\mu(Q)}\right)^{1/p}\frac{1}{\mu(x_Q)}\mu(x_Q)\abs{b(x_Q)}\\
&\\
&=\left(\frac{a(\epsilon,\delta_1)}{\delta_1}\right)^{1/p}+\left(\frac{c(\epsilon,\delta_1,\delta_2)}{\delta_1}\right)^{1/p}\frac{1}{\delta_1-c(\epsilon,\delta_1,\delta_2)}\norm{b}_{L^1(\Omega)},
\end{split}
\end{equation*}
where we used absolute continuity together with the imposed condition $c(\epsilon,\delta_1,\delta_2)\leq \delta_2$ on the second of the three terms. Finally, we impose further conditions on $a(\epsilon,\delta_1)$ by choosing it so that
$$
\left(\frac{a(\epsilon,\delta_1)}{\delta_1}\right)^{1/p}:= \epsilon
$$
and on $c(\epsilon,\delta_1,\delta_2)$ by choosing it so that
$$
\left(\frac{c(\epsilon,\delta_1,\delta_2)}{\delta_1}\right)^{1/p}\frac{1}{\delta_1-c(\epsilon,\delta_1,\delta_2)}\norm{b}_{L^1(\Omega)}\leq \epsilon,
$$
while keeping the already imposed conditions $c(\epsilon,\delta_1,\delta_2)\leq \delta_2$ and $c(\epsilon,\delta_1,\delta_2)< \delta_1$.
\end{proof}

\subsection{John--Nirenberg inequality for the masswise VMO}
We begin by setting up the following notation for the $L^\infty(Q)$ norm modulo the worst exceptional set $E$ of measure $\mu(E)\leq\gamma\mu(Q):$
\begin{equation*}
\begin{split}
\norm{f}_{L^\infty_\gamma(Q)} &:=(f1_Q)^*(\gamma \mu(Q)) \\&:=\inf \big\{B>0 : \mu(Q\cap \{|f|>B\})\leq \gamma \mu(Q)\big\} \\ 
&=\inf_{E : \mu(E\cap Q) \leq \gamma\mu(E)}\norm{f}_{L^\infty(Q\setminus E)}.
\end{split}
\end{equation*}
In particular, 
$$
\norm{f}_{L^\infty(Q)}=\norm{f}_{L^\infty_0(Q)}\geq \norm{f}_{L^\infty_{\gamma_1}(Q)}\geq \norm{f}_{L^\infty_{\gamma_2}(Q)}\geq \norm{f}_{L^\infty_1(Q)}=0,
$$ 
for $0\leq \gamma_1\leq \gamma_2\leq 1$. Moreover, by Chebyshev's inequality
\begin{align}\label{eq:Chebaby}
    \norm{f}_{L^\infty_\gamma(Q)}\leq \phi^{-1}\left(\frac{1}{\gamma} \frac{1}{\mu(Q)}\int_Q \phi(f)\dmu\right),
\end{align} 
for every invertible increasing function $\phi:[0,+\infty)\to[0,+\infty)$ with $\phi(0)=0$. Similarly,
\begin{equation}\label{eq:weakl1}
 \norm{f}_{L^\infty_\gamma(Q)}\leq \frac{1}{\gamma}\norm{f}_{\textit{\L}^{1,\infty}(Q)}.
\end{equation}

 The exponential integrability of functions of bounded mean oscillation was established by John and Nirenberg \cite{johnnirenberg1961}. This together with the works by John \cite{john1965} and Str\"omberg \cite{stromberg1979} yields the following John--Nirenberg-type inequality:
\begin{equation}\label{eq:john_nirenberg_stromberg}
\frac{\big|Q\cap \{\abs{b-\angles{b}_Q}>\lambda\}\big|}{|Q|}\lesssim \exp(-c (\sup_{Q} \norm{b-\angles{b}_Q}_{L^\infty_\gamma(Q)})^{-1} \lambda)
\end{equation}
for $\gamma\in(0,1/2]$.
These and related inequalities in varied contexts are all known as John--Nirenberg-type inequalities. The area of such inequalities is vast; for illustration, the original paper by John and Nirenberg \cite{johnnirenberg1961} is cited by hundreds of papers. For further references, see for example Grafako's textbooks \cite{GrafakosCFA,GraMFA}.

We move on to state a dyadic version of the John--Nirenberg-type inequality, and for this recall some standard dyadic notation.
We call a collection $\cq$ of sets {\it dyadic} if for every $Q,R\in \cq$ with $Q\cap R\neq \emptyset$ we have $Q\subseteq R$ or $R\subseteq Q$. As usual in the context of stopping cubes, the parent and the children relative to a dyadic collection $\cq$ are defined by
$$
\piq(Q):=\min \{R\in\cq : R\supsetneq Q\},\quad 
\chq(Q):=\{R\in \cq : \text{$R$ maximal with $R\subsetneq Q$}\}.
$$
The notation $\cq(Q):=\{R\in \cq : R\subseteq Q\}$ for the collection of subcubes is used. 

The following John--Nirenberg-type inequality is a slight variant of the inequality \eqref{eq:john_nirenberg_stromberg} and of the John--Nirenberg inequality for martingale $\operatorname{BMO}$ spaces \cite{herz1974} (in the particular case of the dyadic filtration). 
The proof below is a combination of the proof of Lerner's median oscillation decomposition \cite{Ler10} with a stopping condition from \cite{Han15}.
\begin{lemma}[Dyadic John--Nirenberg-type inequality, cf. \cite{johnnirenberg1961,john1965,herz1974,stromberg1979,Ler10}]\label{lem:JN1}Let $(\Omega,\mu)$ be a measure space. Let $\cq$ be a finite dyadic collection of measurable sets. Let $\{\lambda_Q\}_{Q\in\cq}$ be a family of complex-valued measurable functions such that $\supp (\lambda_Q)\subseteq Q$ and such that $\lambda_Q$ is constant on each $R\in\chq(Q)$. 
Then, there holds that 
\begin{equation}\label{eq:temp5}
\abs{\sum_{R\in\cq(Q)}\lambda_R}\lesssim \Big(\sup_{R\in\cq(Q)} \norm{\lambda_R}_{L^\infty(R,\mu)}+\sup_{R\in\cq(Q)} \norm{\sum_{S\in\cq(R)}\lambda_S}_{L^\infty_{1/4}(R,\mu)}\Big)\phi_Q,
\end{equation}
where $\phi_Q:\Omega\to [0,+\infty)$ is a measurable function supported on $Q$ and exponentially integrable in that $\mu(\phi_Q>\lambda)\lesssim 2^{-\lambda}\mu(Q)$.
\end{lemma}
\begin{proof}Proof by stopping time.  The standing assumption in the proof is that all the cubes belong to the collection $\cq$ and hence ``$\cq$'' is mostly omitted from the notation. Set 
\begin{align*}
    C &:=\sup_{Q\in\cq} \norm{\lambda_Q}_{L^\infty}, \qquad B:=\sup_{Q\in\cq} B_Q,\\
B_Q &:=\Big\|\sum_{R\subseteq Q} \lambda_R\Big\|_{L^\infty_{1/4}(Q)}:=\min\Big\{B>0 : \mu\big(\abs{\sum_{R\subseteq Q} \lambda_R }>B\big)\leq \frac{1}{4}\mu(Q) \Big\},
\end{align*}
Fix a cube $F\in\cq$. Let $\chf(F):=\{F'\}$ be the set of all the maximal cubes $F'\in\cq(F)$ that satisfy the stopping condition
$$
\abs{\sum_{F'\subsetneq Q\subseteq F}\lambda_Q}>2B \quad\text{on $F'$},
$$ 
notice that the function on the left-hand side is constant on $F'$. 
Since
$$
1_{F'}\sum_{F'\subsetneq Q\subseteq F}\lambda_Q=1_{F'}\Big(
\sum_{Q\subseteq F}\lambda_Q-\sum_{Q\subseteq F'}\lambda_Q\Big),
$$
we have
\begin{align*}
    \mu(F') &=\mu(F'\cap\{\abs{\sum_{F'\subsetneq Q\subseteq F}\lambda_Q}>2B\} \\ 
    &\leq \mu(F'\cap \{\sum_{Q\subseteq F}\lambda_Q>B\})+  \mu(F'\cap \{\sum_{Q\subseteq F'}\lambda_Q>B\}).
\end{align*}
Summing this over the disjoint cubes $F'$ yields
\begin{align}\label{eq:childrenofthelight}
    \mu(F')\leq \mu(F\cap \{\sum_{Q\subseteq F}\lambda_Q>B_F\})+  \sum_{F'}\mu(F'\cap \{\sum_{Q\subseteq F'}\lambda_Q>B_{F'}\})\leq \frac{1}{2}\mu(F).
\end{align}
Set $E(F):=F\setminus \bigcup_{F'} F'$. We decompose
$$
\sum_{Q\subseteq F}\lambda_Q=\sum_{F'} 1_{F'} \Big(\sum_{\pi(F')\subsetneq Q\subseteq F}\lambda_Q\Big)+ 1_{E(F)} \sum_{Q\subseteq F}\lambda_Q+\sum_{F'}1_{F'}\lambda_{\pi(F')}  +\sum_{F'}\sum_{Q\subseteq F'}\lambda_Q.
$$
The cube $\pi(F')\supsetneq F'$ does not satisfy the stopping condition and hence
$$
\abs{\sum_{\pi(F')\subsetneq Q\subseteq F}\lambda_Q}\leq 2B.
$$
Next, consider $x\in E(F)$ and let $Q_*(x)$ denote the minimal cube $Q\in \cq$ such that $x\in Q\subseteq F$. If $Q_*(x)=F$, then
$
1_{E(F)}(x) \abs{\sum_{Q\subseteq F}\lambda_Q} = |\lambda_F|\leq C1_F.
$
If $Q_*(x)\subsetneq F$, then $Q_*(x)$ does not satisfy the stopping condition because otherwise there would be a cube $F'$ such that $F'\supseteq Q_*(x)$ and hence $x\notin E(F)$, which would be a contradiction. Therefore,
$$
1_{E(F)}(x) \abs{\sum_{Q\subseteq F}\lambda_Q}=1_{Q_*}\abs{\sum_{Q_*(x)\subsetneq Q\subseteq F}\lambda_Q}\leq 2B1_F.
$$
Altogether, 
$$
\abs{\sum_{Q\subseteq F}\lambda_Q}\leq (2B+C) 1_{F} +(2B+C)\sum_{F'} 1_{F'}+\sum_{F'}1_{F'}\abs{\sum_{Q\subseteq F'}\lambda_Q}.
$$
The proof is completed by using the above pointwise estimate and the estimate \eqref{eq:childrenofthelight} for the measure of the stopping children iteratively.
\end{proof}
\begin{remark}The above proof works in the following slightly more general cases: (1) In the case of an infinite collection $\cq$, assuming the existence of unconditional sums and the existence of maximal subsets. (2) In the case of vector-valued functions (instead of complex-valued). (3) In the case that for some $k\in\bn$ every function $\lambda_Q$ is constant on each $Q'\in\chq^{(k)}(Q)$, after replacing the term $\sup_Q \norm{\lambda_Q}_{L^\infty(Q)}$ by the term $k\cdot \sup_Q \norm{\lambda_Q}_{L^\infty(Q)}$.
\end{remark}
As well-known, the term $\sup_{R\in\cq(Q)} \norm{\lambda_R}_{L^\infty(R,\mu)}$ can be essentially omitted from \eqref{eq:temp5} when the underlying measure is doubling:
\begin{lemma}[Estimate when the underlying measure is doubling]\label{lem:JN2} Let $\cd$ be a dyadic system and let $ \cq\subseteq\cd$ be a collection of dyadic cubes. Assume that $\mu$ is dyadically doubling, which means (written in terms of the inverse of the usual quantity) that
$$
\gamma_\mu:=\inf_{Q\in\cq}\inf_{Q'\in\chd(Q)} \frac{\mu(Q')}{\mu(Q)}>0.
$$
Assume that every function $\lambda_Q$ is constant on each $Q'\in\chd(Q)$. 
Then for every $Q\in\cq$ there holds that
$$
\norm{\lambda_Q}_{L^\infty(Q,\mu)}\leq 2 \sup_{R\in\cq(Q)} \norm{\sum_{S\in\cq(R)}\lambda_S}_{L^\infty_{\gamma_\mu/2}(R,\mu)}.
$$
\end{lemma}
 \begin{proof}
Fix $Q\in \cq$. Set 
$$
B(\lambda_Q):=\inf\{B>0 : \mu(\lambda_Q>B)\leq \gamma_\mu \mu(Q)\}.
$$ 
We claim that $B(\lambda_Q)\geq \norm{\lambda_Q}_\infty.$ 
Let $\varepsilon>0$ be an auxiliary number. 
Since $\lambda_Q$ is supported on $Q$, constant on each $Q'\in\chd(Q)$, and $Q=\bigcup_{R\in \chd(Q)}Q'$, we have
$$
\mu(\lambda_Q>\norm{\lambda_Q}_\infty -\varepsilon)\geq \inf_{Q'\in\chd(Q)} \mu(Q')\geq \gamma_\mu \mu(Q).
$$
Therefore
\begin{align*}
    &\inf \{B>0 : \mu(\lambda_Q>B)\leq \gamma_\mu(Q)-\varepsilon\} \\ 
    &\qquad \geq \inf \{B>0 : \mu(\lambda_Q>B)< \gamma_\mu(Q)\} \geq \norm{\lambda_Q}_\infty -\varepsilon
\end{align*}
and hence, by taking the limit $\varepsilon\to 0:$ 
$$
B(\lambda_Q)\geq \norm{\lambda_Q}_\infty.
$$
Now, set
$$
B(\sum_{R\in\cq(Q)}\lambda_R):=\inf\big\{B>0 : \mu\big(\sum_{R\in\cq(Q)} \lambda_R>B\big)\leq \frac{1}{2}\gamma_\mu \mu(Q)\big\}.
$$
Now
\begin{align*}
    &\mu\Big(\abs{\lambda_Q}> B(\sum_{R\in\cq(Q)}\lambda_R)+\sup_{R\in\chq(Q)}B(\sum_{S\in\cq(R)}\lambda_S)\Big)\\
&=\mu\Big(\abs{\sum_{R\in\cq(Q)}\lambda_R-\sum_{R\in\chq(Q)}\sum_{S\in\cq(R)}\lambda_S}> B(\sum_{R\in\cq(Q)}\lambda_R)+\sup_{R\in\chq(Q)}B(\sum_{S\in\cq(R)}\lambda_S)\Big)\\
&\leq \mu\Big(\abs{\sum_{R\in\cq(Q)}\lambda_R}> B(\sum_{R\in\cq(Q)}\lambda_R)\Big) \\
&\qquad\qquad+\mu\Big(\abs{\sum_{R\in\chq(Q)}\sum_{S\in\cq(R)}\lambda_S}> \sup_{R\in\chq(Q)}B(\sum_{S\in\cq(R)}\lambda_S)\Big)\\
&\leq \mu\Big(\abs{\sum_{R\in\cq(Q)}\lambda_R}> B(\sum_{R\in\cq(Q)}\lambda_R)\Big) \\ 
&\qquad\qquad+\sum_{R\in\chq(Q)}\mu\Big(\abs{\sum_{S\in\cq(R)}\lambda_S}> B(\sum_{S\in\cq(R)}\lambda_S)\Big)\\
&\leq \frac{1}{2}\gamma_\mu\mu(Q)+\sum_{R\in\chq(Q)}\frac{1}{2}\gamma_\mu\mu(R)\leq \gamma_\mu\mu(Q).
\end{align*}
Therefore,
$$
B(\lambda_Q)\leq B(\sum_{R\in\cq(Q)}\lambda_R)+\sup_{R\in\chq(R)}B(\sum_{S\in\cq(R)}\lambda_S)\leq 2 \sup_{R\in\cq(Q)}B(\sum_{S\in\cq(R)}\lambda_S).
$$
 \end{proof}
 Combining the above two lemmas yields immediately the following inequality:
\begin{corollary}[John--Nirenberg inequality for collections of dyadic cubes]\label{corollary:john_nirenberg}\label{lem:JN3}Let $\cd$ be a dyadic system and $\cq\subseteq \cd$ a collection of dyadic cubes. Assume that $\mu$ is dyadically doubling. Let $\{\lambda_Q\}_{Q\in\cq}$ be a family of functions such that every function $\lambda_Q$ is supported on $Q$ and constant on each $Q'\in \chd(Q)$. Then
$$
\sup_{Q\in\cq}\frac{1}{\mu(Q)^{1/p}}\norm{\sum_{R\in \cq: R \subseteq Q}\lambda_R}_{L^p(Q,\mu)}\lesssim_{\mu,p}\sup_{Q\in\cq}\norm{\sum_{R\in \cq: R \subseteq Q}\lambda_R}_{\textit{\L}^{1,\infty}(Q,\mu)}
$$
for all $p\in[1,+\infty)$. The implicit constant depends only on the doubling constant of the measure $\mu$.
\end{corollary}

\begin{proof} For the reader's convenience, the well-known estimation is presented in what follows. Fix $Q\in\cq$. 
By Lemma \ref{lem:JN1} we bound 
\begin{align*}
    \norm{\sum_{R\in\cq(Q)}\lambda_R}_{L^p(Q,\mu)} \lesssim \Big( \sup_{R\in\cq(Q)} \norm{\lambda_R}_{L^\infty(R,\mu)} + \sup_{R\in\cq(Q)}
     \norm{\sum_{S\in\cq(R)}\lambda_S}_{L^\infty_{1/4}(R,\mu)}  \Big)\norm{\phi_Q}_{L^p(\mu)}. 
\end{align*}
The first bracketed term can be estimated as follows. Since $\lambda_Q$ is constant on the dyadic children $Q'\in\chd(Q)$ and $\mu$ is dyadically doubling, we obtain
\begin{align*}
     \norm{\lambda_R}_{L^\infty(R)} &= \max_{R'\in\chd(R)} \big\| \lambda_R\big\|_{\textit{\L}^{1,\infty}(R',\mu)}  \\
     &= \max_{R'\in\chd(R)} \Big\|  \sum_{S\in \cq: S \subseteq R}\lambda_S -  \sum_{S\in \cq: S \subseteq R'}\lambda_S \Big\|_{\textit{\L}^{1,\infty}(R',\mu)}  \\
    &\lesssim_\mu \sup_{S\in\cq(R)} 
    \Big\|  \sum_{T\in\cq(S)}\lambda_T \Big\|_{\textit{\L}^{1,\infty}(S,\mu)}.
\end{align*}
(Alternatively, the first bracketed term can be essentially absorbed into the second  bracketed term by Lemma \ref{lem:JN2}.) For the second bracketed term it follows by \eqref{eq:weakl1} that
$$
\Big\|\sum_{S\in\cq(R)}\lambda_S\Big\|_{L^\infty_{1/4}(R,\mu)}\lesssim \Big\|  \sum_{S\in\cq(R)}\lambda_S \Big\|_{\textit{\L}^{1,\infty}(R,\mu)}.
$$
Finally, we notice that
\begin{align*}
    \norm{\phi_Q}_{L^p(\mu)}^p &=p\int_0^\infty \lambda^{p-1}\mu(\phi_Q>\lambda)\,\mathrm{d}\lambda \\&\leq p \mu(Q) \int_0^\infty \lambda^{p-1}\min\{1,2^{-\lambda}\}\,\mathrm{d}\lambda\lesssim_p \mu(Q).
\end{align*}

\end{proof}
The John--Nirenberg inequality for $\operatorname{VMO}$ on the Lebesgue measure spaces (with Muckenhoupt weights) was proved in \cite{HOS2023} by sparse dominating the oscillation (with the Lebesgue measure) $\fint_Q|b-\langle b\rangle_Q|.$ 
We next prove the John--Nirenberg equality for the measure-theoretic class $\mathrm{VMO}^p(\Omega,\cd,\mu)$ under the assumption that the measure $\mu$ is doubling by using a reduction to the John--Nirenberg inequality for collections of dyadic cubes (Corollary \ref{corollary:john_nirenberg}).

\begin{corollary}[John--Nirenberg inequality for $\operatorname{VMO}$ with doubling measures]\label{cor:JNIElast} Assume that $\mu$ is dyadically doubling. Then
$$
\mathrm{VMO}^p(\Omega,\cd,\mu)=\mathrm{VMO}^1(\Omega,\cd,\mu)
$$
for every $p\in(1,+\infty)$. 
\end{corollary}
\begin{proof} Let $p\in(1,+\infty)$. It is immediate from the proof of sufficiency that the (seemingly slightly weaker) condition $$
\lim_{i\to\infty}\sup_{Q\in\cq_i} \frac{1}{\mu(Q)^{1/p}} \norm{\sum_{\substack{R\in \cq_i : \\R\subseteq Q}}D_Rb}_{L^p(\mu)}=0\quad\text{for  $\cq_i\in\{\cq^{\mathrm{heavy}}_i,\cq^{\mathrm{light}}_i,\cq^{\mathrm{distant}}_i\}$}
$$
is sufficient for the compactness of the dyadic paraproduct $P_{b,\cd}:L^p(\mu)\to L^p(\mu)$, which in turn implies the necessary condition
$$
\lim_{i\to\infty}\sup_{Q\in\cq_i} \frac{1}{\mu(Q)^{1/p}} \norm{b-\angles{b}_Q}_{L^p(\mu)}=0\quad\text{for  $\cq_i\in\{\cq^{\mathrm{heavy}}_i,\cq^{\mathrm{light}}_i,\cq^{\mathrm{distant}}_i\}$}.
$$
Therefore, it suffices to check that
$$
\sup_{Q\in\cq} \frac{1}{\mu(Q)^{1/p}} \norm{\sum_{R\in \cq: R \subseteq Q}D_Rb}_{L^p(\mu)}\lesssim_{p,\mu} \sup_{Q\in\cq} \frac{1}{\mu(Q)} \norm{b-\angles{b}_Q}_{L^1(\mu)}
$$
for an arbitrary subcollection $\cq\subseteq \cd$. This is checked in two steps. First, by the John--Nirenberg inequality  (Corollary \ref{corollary:john_nirenberg}),
$$
\sup_{Q\in\cq} \frac{1}{\mu(Q)^{1/p}} \norm{\sum_{R\in \cq: R \subseteq Q}D_Rb}_{L^p(\mu)}\lesssim_{p,\mu} \sup_{Q\in\cq} \frac{1}{\mu(Q)} \norm{\sum_{R\in \cq: R \subseteq Q}D_Rb}_{L^{1,\infty}(\mu)}.
$$
Then, by Burkholder's weak-$L^1$ inequality, 
$$
\frac{1}{\mu(Q)} \norm{\sum_{R\in \cq: R \subseteq Q}D_Rb}_{L^{1,\infty}(\mu)}\lesssim \frac{1}{\mu(Q)} \norm{\sum_{R\in \cd: R \subseteq Q}D_Rb}_{L^1(\mu)}=\frac{1}{\mu(Q)} \norm{b-\angles{b}_Q}_{L^1(Q,\mu)}.
$$
\end{proof}

\printbibliography
\end{document}